\documentclass[a4paper]{scrartcl}

\usepackage{graphicx}
\usepackage[utf8]{inputenc}
\usepackage[english]{babel}
\usepackage{amsmath,amssymb,enumerate,url,appendix,tabularx,stmaryrd,mathrsfs}

\usepackage{todonotes,comment,afterpage}
\usepackage{subcaption}
\usepackage{framed}
\usepackage{mathtools}

\usepackage{amsthm}
\usepackage{graphicx}

\theoremstyle{plain}
\newtheorem{theorem}{Theorem}[section]
\newtheorem{lemma}[theorem]{Lemma}
\newtheorem{corollary}[theorem]{Corollary}
\newtheorem{proposition}[theorem]{Proposition}

\newtheorem{assumption}[theorem]{Assumption}

\newtheorem{definition}[theorem]{Definition}
\newtheorem{example}[theorem]{Example}

\newtheorem{remark}[theorem]{Remark}

\numberwithin{equation}{section}
\numberwithin{figure}{section}
\numberwithin{table}{section}

\newcommand{\R}{\mathbb{R}}
\newcommand{\C}{\mathbb{C}}
\newcommand{\N}{\mathbb{N}}

\newcommand{\bigO}{\mathcal{O}}

\newcommand{\eremk}{\hbox{}\hfill\rule{0.8ex}{0.8ex}}

\newcommand{\abs}[1]{\left|#1\right|}
\newcommand{\norm}[1]{\left\|#1\right\|}

\newcommand{\dom}{\operatorname{dom}}

\def\TT{\mathcal{T}}
\def\SS{\mathcal{S}}

\def\QQ{Q}
\def\ii{i}

\def\LL{\mathcal{L}}
\def\CC{\mathcal{C}}

\def\sector{\mathscr{S}}

\newcommand{\HHx}{\mathbb{V}_{h}}

\renewcommand{\Re}{\operatorname{Re}}
\renewcommand{\Im}{\operatorname{Im}}

\def\Nquad{\mathcal{N}_{\textrm{q}}}
\def\kquad{k}
\def\Nhpquad{\mathcal{N}_{\textrm{hp}}}
\def\Ndof{\mathcal{N}_{\Omega}}

\def\uu{u}
\def\uqh{u^{q,h}}
\def\uq{u^{q}}
\def\uui{u_i}
\def\uiq{u_i^{q}}
\def\uiqq{u_i^{q,q}}
\def\uiqqh{u_i^{q,q,h}}

\def\uufd{u^{\mathrm{fd}}}

\iftrue 
\usepackage{tikz}
\usetikzlibrary{spy}

\usetikzlibrary{shapes}
\usepackage{pgfplots}
\usetikzlibrary{calc,spy}
\newcommand{\includeTikzOrEps}[1]{ \include{figures/#1}}  

\else
\newcommand{\includeTikzOrEps}[1]{\includegraphics{figures_pdf/#1}}
\fi

\begin{document}
\title{ An exponentially convergent discretization for space-time fractional parabolic equations using $hp$-FEM  }

\author{%
  Jens Markus Melenk\thanks{Institute for Analysis and Scientific Computing, Technische Universit\"at Wien,
    Wiedner Hauptstrasse 8-10, 1040 Vienna, Austria. melenk@tuwien.ac.at} and %
  Alexander Rieder\thanks{ Institute for Analysis and Scientific Computing, Technische Universit\"at Wien,
   Wiedner Hauptstrasse 8-10, 1040 Vienna, Austria. alexander.rieder@tuwien.ac.at}  
 }


\date{\today}

\maketitle
{\vspace{5mm}\centering \textbf{  Dedicated to Professor Christoph Schwab on the occasion of his $60^{th}$ birthday.\vspace{5mm}}\par}

\begin{abstract}
  {We consider a space-time fractional parabolic problem. Combining a sinc-quadrature based method for discretizing
  the Riesz-Dunford integral with $hp$-FEM in space yields an exponentially convergent scheme for the initial boundary value problem
  with homogeneous right-hand side. For the inhomogeneous problem, an $hp$-quadrature scheme is implemented.
  We rigorously prove exponential convergence with focus on small times $t$, proving robustness with respect to
  startup singularities due to data incompatibilities.}
{fractional diffusion, sinc-quadrature, Mittag-Leffler, Riesz-Dunford, $hp$-FEM}
\end{abstract}

\section{Introduction}
The study of fractional partial differential equations has attracted a lot of interest in the mathematical community in recent years.
Motivated by processes in physics or finance, it has become necessary to leave the realm of classical derivatives, and one encounters  new
difficulties, most notably one introduces non-local aspects to the problems under consideration~\cite{collection_of_applications}.

In the context of numerical approximation, several different approaches have been proposed to handle fractional diffusion problems,
although often they focused on the stationary elliptic case.
For the stationary case, we mention those based on a Caffarelli-Silvestre extension~\cite{nos15,tensor_fem,mpsv18}, the integral fractional Laplacian~\cite{ab17,fmms21}
and the Riesz-Dunford (also known as Riesz-Taylor and Dunford-Taylor) functional calculus~\cite{bonito_pasciak_parabolic,blp17,tensor_fem_on_polygons}.
See~\cite{bbnos18} for a summary of discretization methods for the elliptic fractional diffusion problem.
Another recent approach for discretizing both the fractional Laplace and  heat equation is based on a reduced basis method~\cite{danczul2019reduced,danczul2020reduced}.
Further methods that are based on discretizing first and computing fractional powers of the resulting stiffness
  matrix include~\cite{HHT_08,GHK_05,HLMMV_18,HLMMP_20}; see also the discussion in~\cite{ClHofreit}.

For fractional ODE problems, the most common formulation is based on the Caputo derivative due to its natural behavior with respect to initial conditions.
As for numerical approximation, it is common to use time-stepping methods~\cite{noes16}, but  especially when already dealing with fractional operators in space,
it is very natural to use a Riesz-Dunford based formulation and apply an appropriate sinc-quadrature~\cite{bonito_pasciak_parabolic}. Recently, a slight modification of the sinc-quadrature scheme based on a double exponential
  transformation has been proposed~\cite{de_quad}. This modified scheme can also be combined
  with the $hp$-FEM techniques from this article to obtain a very fast numerical scheme.

For the spatial discretization, finite element based-methods are popular. Most results focus on the $h$-version,
designing schemes which provide algebraic convergence rates, under suitable compatibility conditions on the given data.
Lately, $hp$-FEM based schemes have started to appear. First, only for the extended variable in a  Caffarelli-Silvestre based discretization scheme~\cite{mpsv18},
then for the full discretization of an elliptic fractional problem~\cite{tensor_fem,tensor_fem_on_polygons}. Unlike previous results, \cite{tensor_fem} also removed the
compatibility conditions on the initial data, relying instead on the fact that appropriately designed $hp$-FEM spaces can resolve the developing singularities at the boundary.

In the context of time-dependent problems, $hp$-FEM based approaches have been pioneered in~\cite{hp_for_heat}, again focusing on an extension-based formulation for the fractional Laplacian
and restricted to the classical first order derivative in time. It is proven for smooth data and geometries (but without compatibility assumption) that an $hp$-discretization can
deliver exponential convergence towards the exact solution.
The present work consists of transferring the methods from~\cite{hp_for_heat} to a functional calculus based discretization in the spirit of~\cite{bonito_pasciak_parabolic}, generalizing
the problem class to the case of (Caputo) fractional time derivatives in the process. We prove that for smooth geometries in 1D or 2D, one can design meshes and approximation spaces
such that the proposed method features exponential convergence to the exact solution, even in the presence of startup singularities.
Most of the paper is concerned with analyzing the convergence in the pointwise-in-time $\widetilde{H}^{\beta}(\Omega)$ Sobolev-norm, which is the natural setting for the model problem~\eqref{eq:model_problem} and
the considered numerical method. For small times $t$, these estimates suffer from a degeneracy at small times $t>0$. In order to remedy this, we consider an appropriate space-time energy norm and prove that, given an additional abstract assumption on the initial condition, one can expect robust convergence in this  weaker norm.

Compared to~\cite{blp17}, we also improve the time dependence of the estimates for the sinc-quadrature error from $t^{-\gamma}$
to $t^{-\gamma/2}$. Similar behavior for the discretization errors was also observed in the $\gamma=1$ case in~\cite{hp_for_heat} and
is well established for discretizations of the heat-equation~\cite{thomee_book}.

We close with a short comment on notation. We write $a \lesssim b$ if there exists a constant $C>0$, which is independent of the main quantities of interest (for example the number of quadrature
points, the mesh size, polynomial degree employed or the time $t$) such that $a \leq C b$. We write $a \sim b$ if $a \lesssim b$ and $b \lesssim a$. The specific dependencies of the implied constants
are stated in the context.

\section{Model Problem and notation}
We consider the  numerical approximation of the following time-dependent problem.
Working on a bounded Lipschitz domain $\Omega \subseteq \R^d$,
we fix $\gamma, \beta \in (0,1]$ and a  final time $T>0$.
Given an initial condition $u_0 \in L^2(\Omega)$
and right-hand side $f \in L^{\infty}((0,T),L^2(\Omega))$ we seek $u:[0,T]\to \R$ satisfying
  \begin{align}
  \label{eq:model_problem}
    \partial^{\gamma}_t u + \LL^{\beta} u &= f \;\, \text{in } \Omega \times [0,T],   \qquad
    u|_{\partial \Omega}=0,  \qquad
    u(0)=u_0 \;\, \text{in } \Omega .
  \end{align}
The operator $\LL u :=-\operatorname{div}(A \nabla u) + c u $ is linear, elliptic and self adjoint. We assume that the given coefficients satisfy that
$A \in L^{\infty}(\Omega,\R^{d\times d})$ is uniformly symmetric, positive definite and $c \in L^{\infty}(\Omega)$ satisfies $c\geq 0$.
The fractional power $\LL^\beta$ is given using the spectral representation
\begin{align}
  \LL^\beta u:=\sum_{j=0}^{\infty} \lambda_j^\beta (u,\varphi_j) \varphi_j,
\end{align}
where $(\lambda_j,\varphi_j)_{j=0}^{\infty}$ are the eigenvalues and eigenfunctions of the operator $\LL$ with homogeneous Dirichlet boundary conditions;
as is standard, the eigenfunctions are $L^2(\Omega)$-orthonormalized.
The homogeneous Dirichlet boundary condition is to be understood in the sense that $u(t) \in \dom(\LL^{\beta})$.

For $\gamma \in (0,1)$, the fractional time derivative is taken in the sense of Caputo, i.e.,
\begin{align*}
  \partial_t^\gamma u(t) := \frac{1}{\Gamma(1-\gamma)} \int_{0}^{t}{\frac{1}{(t-r)^\gamma} \frac{\partial u(t)}{\partial r} \,dr},
\end{align*}
whereas for $\gamma=1$, $\partial^1_t:=\partial_t$ is the classical derivative.

As was shown in~\cite{bonito_pasciak_parabolic}, the exact solution  to (\ref{eq:model_problem})
can be written in the following form using the
Mittag-Leffler function $e_{\gamma,\mu}$ (see \eqref{eq:mittag_leffler} for the precise definition):
\begin{align*}
  u(t)&=E(t) u_0 + \int_{0}^{t} { W(\tau) f(t-\tau)\,d\tau} := e_{\gamma,1}\big(-t^{\gamma} \LL^{\beta}\big) u_0 + \int_{0}^{t} { \tau^{\gamma-1} e_{\gamma,\gamma}\big(-\tau^{\gamma} \LL^\beta\big) f(t-\tau)\,d\tau},
\end{align*}
where we used the following functional calculus based on the spectral decomposition:
\begin{align*}
  e_{\gamma,\mu}\big(-t^{\gamma} \LL^{\beta}\big)v:=\sum_{j=0}^{\infty}
  {e_{\gamma,\mu}(-t^\gamma \lambda_j^\beta) (v,\varphi_j)_{L^2(\Omega)}\, \varphi_j} \qquad \forall v \in L^2(\Omega).
\end{align*}
An alternative representation, which will prove more amenable to numerical discretization is based on the Riesz-Dunford calculus:
\begin{align}
  \label{eq:riest_dunford}
  e_{\gamma,\mu}\big(-t^{\gamma} \LL^{\beta}\big) v
  &=\frac{1}{2\pi \ii}\int_{\CC}{e_{\gamma,\mu}(-t^\gamma z^\beta) \, \big(z - \LL\big)^{-1} v \,dz} \qquad \forall v \in L^2(\Omega),
\end{align}
where the contour $\CC$ is taken to be parameterized by
\begin{align}
      z(y):=b (\cosh(y) + \ii \sinh(y)) \quad \text{ for $y \in \R$}. \label{eq:def_of_z}
    \end{align}  
The parameter $b>0$ is chosen sufficiently small so that $z(y)$, $y \in \R$, is in the sector $\sector$ given in Definition~\ref{def:domain_of_ellipticity}.

The natural spaces for formulating our results are given by two scales of interpolation spaces, $H^{\theta}(\Omega)$ and $\widetilde{H}^{\theta}(\Omega)$. To that end, we define for $\theta \in [0,1]$:
\begin{align*}
  \mathbb{H}^{\theta}(\Omega)&:=\Big\{ u \in L^2(\Omega): \; \sum_{j=0}^{\infty}{\lambda_j^\theta |(u,\varphi_j)_{L^2(\Omega)}|^2 }< \infty \Big\}.
\end{align*}
Another way of introducing spaces between $L^2$ and $H^1_0$ is given by the real interpolation method. That is, given two
Banach spaces $X_1 \subseteq X_0$
with continuous embedding, we define for $\theta\in (0,1)$:
\begin{align*}
     \norm{u}^2_{[X_0,X_1]_{\theta,2}}
  &:= \int_{t=0}^\infty{ t^{-2\theta} \left( \inf_{v \in X_1} \|u - v\|_{0} + t \|v\|_1\right)^2 \frac{dt}{t}}, \\
    \left[X_0,X_1\right]_{\theta,2}&:=\left\{u \in X_0: \norm{u}_{[X_0,X_1]_{\theta,2}} < \infty \right\}.
\end{align*}
For $\theta=0$ and $\theta=1$ we take the convention that $[X_0,X_1]_{\theta,2}=X_{\theta}$.
Using this notation, we define the fractional Sobolev spaces:
\begin{align*}
  H^\theta(\Omega):=\big[L^2(\Omega), H^1(\Omega)\big]_{\theta,2}, \qquad    \widetilde{H}^\theta(\Omega):=\big[L^2(\Omega), H_0^1(\Omega)\big]_{\theta,2}.
\end{align*}

It is well known, that for $\theta \in [0,1]$ the spaces $\widetilde{H}^\theta(\Omega)$ and $\mathbb{H}^{\theta}(\Omega)$ coincide with equivalent norms; see~\cite{tartar}. We will therefore use whichever definition is more convenient. Most notably we have 
\begin{align*}
  \norm{u}_{\widetilde{H}^{\theta}(\Omega)}^2
  &\sim \sum_{j=0}^{\infty}{\lambda_j^\theta |(u,\varphi_j)_{L^2(\Omega)}|^2 }.
\end{align*}
We will need the following multiplicative interpolation estimate (see~\cite[Sect.~{1.3.3}]{Triebel1999}):
\begin{align}
  \label{eq:interpolation_est}
  \norm{u}_{[X_0,X_1]_{\theta,2}}&\leq C_{\theta} \norm{u}_{X_0}^{1-\theta} \norm{u}_{X_1}^{\theta} \qquad \forall u \in X_1, \theta \in [0,1].
\end{align}
Throughout, we take inner products to be antilinear in the second argument.

\subsection{The Mittag-Leffler function}
Since it plays a major role in the numerical method, we briefly introduce the Mittag-Leffler function and
summarize its most important properties. See for example~\cite[Sect.~{1.8}]{ksh06} for a more detailed treatment.
Given parameters $\gamma>0$, $\mu \in \R$, the Mittag-Leffler function is analytic on $\C$ and given by the power series
\begin{align}
  \label{eq:mittag_leffler}
  e_{\gamma,\mu}(z):=\sum_{n=0}^{\infty}{\frac{z^n}{\Gamma(n\gamma+\mu)}}.
\end{align}

For $0 < \gamma < 1, \mu \in \R$ and $\frac{\gamma \pi}{2} < \zeta < \gamma \pi$, there exists a constant $C$ only depending on $\gamma,\mu,\zeta$ such that
\begin{align}
  \label{eq:mittag_leffler_est}
  \abs{e_{\gamma,\mu}(z)}\leq \frac{C}{1+\abs{z}} \qquad \text{for } \zeta \leq \abs{\operatorname{Arg}{z}} \leq \pi.
\end{align}
For $\gamma=\mu=1$, the Mittag-Leffler function $e_{1,1}$ is the usual exponential function. Estimate~\eqref{eq:mittag_leffler_est}
also holds in this case for $ \pi/2<\zeta < \pi$. 

The derivative of the Mittag-Leffler function can be expressed as:
\begin{align}
  \label{eq:dt_of_ml}
  \frac{d}{dt} e_{\gamma,1}(-t^{\gamma} \lambda^{\beta})
  &=-\lambda^{\beta}\,t^{\gamma-1} e_{\gamma,\gamma}(-t^\gamma \lambda^\beta).
\end{align}

\subsection{ Assumptions on the discretization in space}
When considering the Riesz-Dunford representation of $u$, the contour lies in the set of values for which $\LL - z$ is
invertible. Therefore we consider the set of complex numbers up to a cone
  which contains
an interval $[a,\infty) \subseteq \R$ that in turn contains the eigenvalues of $\LL$. 
\begin{definition}
\label{def:domain_of_ellipticity}
Let $C_P$ denote the Poincar\'e constant  of $\Omega$,
i.e., the smallest constant such that $\|u\|_{L^2(\Omega)} \leq C_{P} \|\nabla u\|_{L^2(\Omega)}$
  for all $u \in H_0^1(\Omega)$.
With $\lambda_{\min}(A)$ the smallest eigenvalue of $A$
and 
fixed $0 < \varepsilon_0 < z_0 \leq \min\left(\frac{\lambda_{\min}(A)}{2 C_{\text{P}}},1\right)^2$, we define 
\begin{align*}
  \sector:=\C \setminus \left[\left\{ z_0 + z: \abs{\operatorname{Arg}(z)}\leq \frac{\pi}{8}, \Re(z)\geq 0 \right\} 
  \cup B_{\varepsilon_0}(0)\right].
\end{align*}
\end{definition}
\begin{remark}
  The set $\sector$ is chosen in such a way that it contains the contour $\mathcal{C}$ used in the functional
  calculus, given by $z(y):=b\big(\cosh(y)+ \ii \sinh(y)\big)$.  See Figure~\ref{fig:contour}.
\eremk
\end{remark}

\begin{figure}
  \center
  \includeTikzOrEps{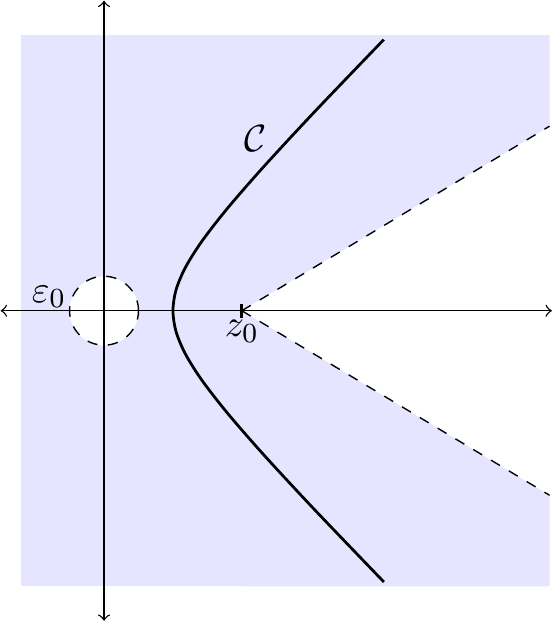}
  \caption{Geometric configuration of Definition~\ref{def:domain_of_ellipticity}}
  \label{fig:contour}
\end{figure}

  An important role will be played by the
  resolvent operator  $R(z):=(z-\LL)^{-1}$
  and its discrete counterpart, where we replace it with a Galerkin solver $R_h(z)$.
  Let $\HHx \subseteq H^1_0(\Omega)$ be a closed subspace.
  Recalling that $\LL u =-\operatorname{div}(A \nabla u) + c u $, 
    we define  $R_h(z) f:=u_h$ as the solution $u_h \in \HHx$ satisfying
  \begin{align}
    \label{eq:def_disc_resolvent}
    \big((z-c) u_h,v_h\big)_{L^2(\Omega)} - \big(A\nabla u_h,\nabla v_h\big)_{L^2(\Omega)}&=\left(f,v_h\right)_{L^2(\Omega)} \qquad \forall  v_h \in \HHx.
  \end{align} 

  For $z$ in the set $\sector$ the following stability estimate holds for the
  resolvent operator:
  \begin{proposition}
    \label{prop:resolvent_estimates_l2}
    Let $z \in \sector$ and $u_0 \in L^2(\Omega)$.
    Then, for $\alpha \in [0,1]$:
    \begin{align}
      \label{eq:resolvent_estimates_l2}
      \norm{R(z) u_0}_{\widetilde{H}^{\alpha}(\Omega)}
      &\lesssim \abs{z}^{-1+\frac{\alpha}{2}} \norm{u_0}_{L^2(\Omega)} \quad  \text{and}\quad
        \norm{R_h(z) u_{h,0}}_{\widetilde{H}^{\alpha}(\Omega)}
        \lesssim \abs{z}^{-1+\frac{\alpha}{2}} \norm{u_{h,0}}_{L^2(\Omega)}.
  \end{align}    
  \end{proposition}
  \begin{proof}
    This is basically~\cite[Lem.~{A.2}]{hp_for_heat} with the substitution $\zeta=-z+c$. We include the proof for completeness.
    
    Writing $u:=R_{h}(z)u_{0}$, and testing in \eqref{eq:def_disc_resolvent} with $v_h:= \overline{\beta} u$  for $\beta \in \C$ with $\abs{\beta}=1$ to be chosen later, we can show
    \begin{multline}
      \Re(\beta)\big(\|{A^{1/2}\nabla u}\|_{L^2(\Omega)}^2
      + \|{c^{1/2} u}\|_{L^2(\Omega)}^2 \big)- \Re(\beta z) \norm{u}^2_{L^2(\Omega)} \\
    \lesssim \abs{(u_{0},u)_{L^2(\Omega)}}
      \leq \big( \abs{z}^{-1/2}\norm{u_{0}}_{L^2(\Omega)}\big)\big( \abs{z}^{1/2}\norm{u}_{L^2(\Omega)}\big).
      \label{eq:resolvent_stability}
    \end{multline}   
    For small $\abs{z}<2z_0$, we can use $\beta=1$ and bound using the Poincar\'e estimate:
    $$
    \abs{z} \norm{u}^2_{L^2(\Omega)}\leq
    2z_0 C_{P}^2 \|\nabla u\|_{L^2(\Omega)}^2 < \frac{1}{2}\|A\nabla u\|_{L^2(\Omega)}^2,
    $$
    to easily get the a priori estimate for $\alpha \in \{0,1\}$.
    
    For larger $\abs{z}>2 z_0$, we may assume that $|\operatorname{Arg}(z)|\geq \delta >0$ for some $\delta >0$.
    (Graphically speaking, we exclude a slightly thinner cone starting at $0$ instead
    of the cone starting at $z_0$.)
    It is then sufficient, if we can pick $\beta$ such that $\Re(\beta) >0$ and
    $\operatorname{Arg}(-z \beta)\in (-\pi/2,\pi/2)$.
    Just as in~\cite[Lem.~{A.2}]{hp_for_heat}
    one can check that if $\Im(z)\leq 0$, $\beta:=e^{-\ii \frac{\pi-\delta}{2}}$ satisfies the
    necessary conditions that $\Re(\beta)>0$ and $-\Re(\beta z)>0$.
    If $\Im(z)\geq 0$, then $\beta:=e^{\ii \frac{\pi-\delta}{2}}$ does.
    This shows as before $\|R_h(z)u_0\|_{H^1(\Omega)} \lesssim |z|^{-1/2} \|u_0\|_{L^2(\Omega)}$ and also $\|R_h(z) u_0\|_{L^2(\Omega)} \lesssim |z|^{-1} \|u_0\|_{L^2(\Omega)}.$
  
    By interpolation, we deduce that for $\alpha \in [0,1]$:
    \begin{align}
      \label{eq:resolven_estimate_int1}
      \norm{R_h(z) u_{0}}_{\widetilde{H}^{\alpha}(\Omega)}\lesssim \abs{z}^{-1+\alpha/2}  \norm{u_{0}}_{L^2(\Omega)}.
    \end{align}

    The discrete estimate follows verbatim.
  \end{proof}

\begin{definition}
  \label{def:uniformly_analytic}
  A function $f: [0,T] \to L^\infty(\Omega)$ is said to be \emph{uniformly analytic} if:
  \begin{enumerate}[(i)]
  \item For all $t\in [0,T]$, $f(t)$ is analytic in a fixed neighborhood  $\widetilde{\Omega}$ of $\overline{\Omega}$;
  \item there exist constants $C_f,\gamma_f >0$, the \emph{analyticity constants of $f$}, such that for all $t\in [0,T]$ and $p \in \N_0$,
    \begin{align*}
      \norm{\nabla^{p} f(t)}_{L^{\infty}(\widetilde{\Omega})}&\leq C_f \gamma^p p !\,;
    \end{align*}
  \item
    \label{it:uniformly_holomorphic_iii}
    
      there exists an open set $\mathcal{O} \subseteq \C$,
       containing a positive sector $\{ z \in \C: z\neq 0, \abs{\operatorname{Arg}{z}} < \delta \}$ for some constant $\delta > 0$,
      and there exists a constant $C_f >0$ such that,
      for each $v \in L^2(\Omega)$, the function
      $t \mapsto (f(t),v)_{L^2(\Omega)}$ has an analytic extension to $\mathcal{O}$
      and
      \begin{align*}
        \sup_{s \in \mathcal{O}} \frac{|(f(s),v)_{L^2(\Omega)}|}{\norm{v}_{L^2(\Omega)}} \leq C_f <\infty.
      \end{align*}
      
  \end{enumerate}
\end{definition}
The main assumption on the discretization space $\HHx$ is taken such that it approximates the solutions to  singularly perturbed problems exponentially well.
The same assumption has already been used in~\cite{hp_for_heat}.
\begin{assumption}
  \label{ass:approx_of_HHx}  
  A function space $\HHx$ is said to resolve down to the scale $\varepsilon >0 $ if
  for all $z \in \sector$ with $\abs{z}^{-1/2} \geq \varepsilon$  
  and for all $f \in L^2(\Omega)$ that are analytic on a fixed neighborhood $\widetilde{\Omega}$ of $\overline{\Omega}$, 
  the solutions to the elliptic problem
  \begin{align*}
    -z^{-1} \LL u  + u &= f
  \end{align*}
  can be approximated exponentially well from it. That is, 
  there exist constants $C(f)$, $\omega$ and $\mu>0$ such that
  \begin{align*}
    \inf_{v_h \in \HHx}\left[
    \abs{z}^{-1}\norm{\nabla u -  \nabla v}^2_{L^2(\Omega)} + \norm{u-v}^2_{L^2(\Omega)}\right]&\lesssim C(f) e^{-\omega \Ndof^\mu},
  \end{align*}
  where $\Ndof:=\operatorname{dim}(\HHx)$.
  The constant $C(f)$ may depend only on $\widetilde{\Omega}$, the analyticity constants of  $f$,
  on $A$, $c$, $\Omega$, $z_0$, and $\varepsilon_0$, while the constants $\omega$ and $\mu$ may depend only on $A$, $c$, $\widetilde{\Omega}$, $\Omega$,
  $z_0$ and $\varepsilon_0$
  Most  notably the constants are independent of $z$, $\varepsilon$, and $\Ndof$.  
\end{assumption}

We state and prove all our results under the general assumption that $\HHx$ resolves specific scales. In Section~\ref{sect:hp_fem}, we will give one 
construction of such a space using $hp$ finite elements; see also~\cite{tensor_fem} and~\cite{tensor_fem_on_polygons} for similar constructions,
focused on real valued parameters $z$.

\section{The pure initial value problem}
\label{sect:homogeneous_problem}
We first focus on discretizing the homogeneous problem, i.e.,  we only consider the case $f=0$. In this case, the exact solution can be written as 
\begin{align}
  \uu(t)  &:=\frac{1}{2 \pi \ii}\int_{\CC}{e_{\gamma,1}(-t^\gamma z^\beta) \, \big(z - \LL\big)^{-1} u_0 \,dz}.
\end{align}
We discretize this function in two steps. First, we replace the contour integral with a sinc-quadrature rule, analogous to what was done in~\cite{bonito_pasciak_parabolic}.
For a fixed number of quadrature points $\Nquad \in \N$ and grid size $\kquad>0$, we write, setting $y_n:=n k$:
\begin{align}
  \uq(t)&:=\frac{k}{2\pi \ii} \sum_{n=-\Nquad}^{\Nquad}{e_{\gamma,1}\big(-t^\gamma z(y_n)^\beta\big) \, z'(y_n)\big(z(y_n) - \LL\big)^{-1} u_0}.
\end{align}
As a second step, we replace the resolvent operator $R(z)=(z-\LL)^{-1}$ by the discrete
counterpart $R_h(z)$ as defined in~\eqref{eq:def_disc_resolvent}. We obtain the fully discrete approximation
\begin{align}
  \label{eq:fully_discrete}
    \uqh(t)&:=\frac{k}{2\pi \ii } \sum_{n=-\Nquad}^{\Nquad}{e_{\gamma,1}\big(-t^\gamma z(y_n)^\beta\big) \, z'(y_n) R_h\big(z(y_n)\big) u_0}.
\end{align}

\begin{lemma}
  \label{lemma:hom:est_space_discr}
  Assume that $u_0$ is analytic on a neighborhood of the closure of $\Omega$.
  Let $\Nquad \in \N$, $\kquad>0$ be given. Let $\varepsilon_0 < b<z_0$,
  with $\varepsilon_0$ and $z_0$ as defined in Definition~\ref{def:domain_of_ellipticity} and $b$ as in \eqref{eq:def_of_z}.
  Then, if 
   $\HHx$ resolves the scales down to $\varepsilon = b^{-1/2} e^{- \frac{\Nquad \kquad}{2}} $, the error due to the spatial discretization can be bounded by
  \begin{align}
    \norm{\uq(t) - \uqh(t)}_{\widetilde{H}^{\beta}(\Omega)}\leq C t^{-\gamma/2}\, \Nquad \kquad\,  e^{- \omega \Ndof^\mu}.
  \end{align}
  The constants $C$, $\omega$, $\mu$ depend on the constants from Assumption~\ref{ass:approx_of_HHx} and on the initial condition $u_0$.
\end{lemma}
\begin{proof}
    We note the following, easy to verify estimates, with the generic constants
    only depending on $b$~(as used in the definition of the contour $\CC$): 
    \begin{align}
      \abs{z(y_j)}\leq  b e^{ \abs{j}k}\quad \text{and}  \quad
    \abs{z(y_j)} &\sim \abs{z'(y_j)} \sim e^{ \abs{j}k}. \label{eq:zprime_vs_z}
    \end{align}  
  For fixed $z \in \sector$ with $\abs{z}^{1/2} \leq  \sqrt{b}\, e^{\Nquad \kquad/2}$, Assumption~\ref{ass:approx_of_HHx}
  gives using~\eqref{eq:interpolation_est}:
  \begin{align*}
    \norm{R(z)u_0 - R_h(z)u_0}_{\widetilde{H}^{\beta}(\Omega)}\lesssim \abs{z}^{\beta/2-1}e^{- \omega \Ndof^\mu}.
  \end{align*}

  Thus, using~\eqref{eq:mittag_leffler_est}, as long as $\HHx$ resolves all the scales $|z(y_j)|^{-1/2}$, the error can be estimated by
  \begin{align*}
    \norm{\uq(t) - \uqh(t)}_{\widetilde{H}^\beta(\Omega)}
    &\lesssim \kquad \sum_{n=-\Nquad}^{\Nquad}{\big|{e_{\gamma,1}(-t^\gamma z(y_n)^\beta)}\big|\abs{z'(y_n)} \abs{z(y_n)}^{-1+\beta/2} e^{-\omega \Ndof^\mu}}\\
    &\lesssim \kquad  \sum_{n=-\Nquad}^{\Nquad}{ \frac{1}{1+ t^\gamma \abs{z(y_n)}^\beta} \abs{z(y_n)}^{\beta/2} e^{-\omega \Ndof^\mu}}.    
  \end{align*}
  The following simple calculation then concludes the proof:
  \begin{align}
    \label{eq:bounding_the_integrand}
    \begin{split}
    \frac{1}{1+ t^\gamma \abs{z(y_n)}^\beta} \abs{z(y_n)}^{\beta/2}
    &=\Bigg(\frac{1}{1+ t^\gamma \abs{z(y_n)}^\beta} \Bigg)^{1/2}
      \Bigg(\frac{1}{1+ t^\gamma \abs{z(y_n)}^\beta} \abs{z(y_n)}^{\beta}\Bigg)^{1/2} \\
    &\lesssim \Bigg(\frac{1}{\abs{z(y_j)}^{-\beta}+ t^\gamma }\Bigg)^{1/2} 
    \lesssim t^{-\gamma/2}.      
  \end{split}
  \end{align}  
\end{proof}

Next, we consider the discretization error due to replacing the contour integral by the sinc-quadrature formula. This can be done along
the same lines as in \cite{bonito_pasciak_parabolic}. We take their definition for the class of functions which can be approximated well by
the sinc quadrature.
\begin{definition}
  \label{def:sinc_functions}
  Given $H>0$, we define $S(B_H)$ as the set of functions $f$, defined on $\R$, satisfying the following assumptions:
  \begin{enumerate}[(i)]
  \item
    \label{def:sinc_functions:i}
    $f$ can be extended to an analytic function on the infinite strip
    \begin{align}
      B_H:=\{ z \in \C: \abs{\Im(z)} < H\}
    \end{align}
    that is also continuous on $\overline{B_H}$.
  \item
    \label{def:sinc_functions:ii}
    There exists a constant $C>0$ independent of $y \in \R$ such that
    \begin{align}
      \int_{-H}^{H}{\abs{f(y+\ii w)} \,dw}&\leq C \qquad \forall y \in \R.
    \end{align}
  \item
    \label{def:sinc_functions:iii}
    We have
    \begin{align}
      N(B_H):=\int_{-\infty}^{\infty}{\abs{f(y+\ii H)} + \abs{f(y-\ii H)}\,dy} < \infty.
    \end{align}   
  \end{enumerate}
\end{definition}

The following error estimate is proved in~\cite{sinc_book}, and is also the basis for the convergence result~\cite[Lem.~{4.2}]{bonito_pasciak_parabolic}.
\begin{proposition}[{\cite[Thm.~{2.20}]{sinc_book}}]
  \label{prop:sinc_quad}
  If $f \in S(B_H)$ and $\kquad >0$, then
  \begin{align}
    \abs{\int_{-\infty}^{\infty}{f(x)\,dx} - \kquad \sum_{n=-\infty}^{\infty}{f(\kquad \,n)} } &= \frac{ N(B_H)}{2 \sinh(\pi d /\kquad)} e^{-\pi d/\kquad}.
  \end{align}  
\end{proposition}

The behavior of the exact solution, most notably the blowup of the energy norm for small times is determined by the
regularity and boundary conditions of the initial condition $u_0$, formalized using the spaces $\widetilde{H}^{\theta}(\Omega)$. Even if the
boundary conditions are not met, $u_0$ is slightly better than just $L^2(\Omega)$. This is the content of the following proposition.
\begin{proposition}
  \label{prop:no_bc_for_small_sobolev}
  For $\theta \in [0,1/2)$, and $u_0 \in H^\theta(\Omega)$, we can bound
  \begin{align*}
    \norm{u_0}_{\widetilde{H}^{\theta}(\Omega)} &\leq C \norm{u_0}_{H^{\theta}(\Omega)}.
  \end{align*}
\end{proposition}
\begin{proof}
  For $\theta<1/2$, the two families of spaces $H^{\theta}$ and $\widetilde{H}^{\theta}$ coincide with equivalent norms,
  see~\cite[Sect.~{1.11.6}]{triebel3} or~\cite[Thms.~{3.33}, {B.9}, {3.40}]{McLean2000}).
\end{proof}

\begin{lemma}
  \label{lemma:hom:g_lambda_estimates}
  For $\lambda \geq \lambda_1 >0$, we define the function
  \begin{align}
    g_{\lambda}(y,t):=\frac{1}{2\pi \ii } e_{\gamma,1}(-t^\gamma z^\beta) z^\prime(y)(z(y)-\lambda)^{-1}.
  \end{align}
  Then,  for $H \in (0,\pi/4)$ and $\varepsilon \in (0, \beta/2)$, we have the estimate 
  \begin{align}
    \abs{g_{\lambda}(y,t)}\lesssim t^{-\gamma/2} \lambda^{-\beta/2+\varepsilon} e^{- \varepsilon \Re(y)} \qquad \qquad \forall y \in B_H,
    \label{eq:g_lambda_general_est}
  \end{align}
  with implied constant depending on $\lambda_1$, $\beta$ , $\gamma$, and $H$.

  Additionally, it holds that
  $g_{\lambda} \in S(B_{H})$ with
  $$
  N(B_H) \leq C(\beta,H,b,\varepsilon,\gamma) t^{-\gamma/2} \lambda^{-\beta/2 + \varepsilon} .
  $$
\end{lemma}
\begin{proof}
  We begin by noting the  following estimates for $y \in B_H$,
  with implied constants depending on $\lambda_1, b,$ and $H$:
  \begin{subequations}
    \begin{align}
      \label{eq:resolvent_estimates_1}
      \abs{z(y)} + \abs{z^\prime(y)}&\lesssim e^{\abs{\Re(y)}}, \\   
      \label{eq:resolvent_estimates_2}
      \Re(z(y)^\beta)&\gtrsim e^{\beta\abs{\Re(y)}}, \\  
      \label{eq:resolvent_estimates_3}
      |z^\prime(y)(z(y) - \lambda)^{-1}| &\lesssim 1.
    \end{align}
  \end{subequations}
  The first estimate follows from the definition, the others can be found in \cite[Lem.~{B.1}]{bonito_pasciak_parabolic}
  and were proven in the course of the proof~\cite[Thm.{4.1}]{blp17}.

  We assert the following estimate for $\varepsilon \in [0,\beta/2]$:
  \begin{align*}
    \lambda^{\beta} \abs{z^\prime(y) \left(z(y) - \lambda\right)^{-1}} &\lesssim \lambda^{2\varepsilon} e^{(\beta-2\varepsilon) \abs{\Re(y)}}.
  \end{align*}
  To see this, we compute using estimates \eqref{eq:resolvent_estimates_1} and \eqref{eq:resolvent_estimates_3}: 
  \begin{align*}
    \abs{\lambda \,z^\prime(y) \left(z(y) - \lambda\right)^{-1}}
    &=\abs{z^\prime(y) \left( -1 + \frac{z(y)}{z(y)-\lambda}\right)}
    \lesssim \abs{z^\prime(y)} + \abs{z(y)} \lesssim e^{\abs{\Re(y)}}.
  \end{align*}
  Interpolation with~\eqref{eq:resolvent_estimates_3} then gives the general estimate:
  \begin{align*}
    \lambda^{\beta} \abs{z^\prime(y) \left(z(y) - \lambda\right)^{-1}}
    &= \lambda^{2\varepsilon} \left(\abs{z^\prime(y) \left(z(y) - \lambda\right)^{-1}} \right)^{1-\beta+2\varepsilon}  \left( \lambda \abs{z^\prime(y) \left(z(y) - \lambda\right)^{-1}} \right)^{\beta-2\varepsilon} \\
    &\lesssim \lambda^{2\varepsilon} e^{(\beta-2\varepsilon) \abs{\Re(y)}}.
  \end{align*}

  To show~\eqref{eq:g_lambda_general_est}, we use~\eqref{eq:mittag_leffler_est} to get
  \begin{align*}
    \lambda^{\beta/2}\abs{g_{\lambda}(y,t)}
    &\lesssim
      \Big(\abs{g_{\lambda}(y,t)}\Big)^{1/2}\Big(\lambda^{\beta}\abs{g_{\lambda}(y,t)}\Big)^{1/2} \\
    &\lesssim
      \Big(\frac{1}{1+t^\gamma e^{\beta \abs{\Re(y)}}} \Big)^{1/2}
      \Big(\frac{1}{1+t^\gamma e^{\beta \abs{\Re(y)}}} \lambda^{2\varepsilon} e^{(\beta - 2\varepsilon) \abs{\Re(y)}} \Big)^{1/2}\\
      &\lesssim t^{-\gamma/2}  \lambda^{\varepsilon} e^{-\varepsilon \abs{\Re(y)}} .
  \end{align*}

  From estimate~\eqref{eq:g_lambda_general_est}, we easily deduce parts~(\ref{def:sinc_functions:ii}) and (\ref{def:sinc_functions:iii})
  of Definition~\ref{def:sinc_functions}. Part~(\ref{def:sinc_functions:i}) was already shown in~\cite{bonito_pasciak_parabolic}. 
\end{proof}

\begin{lemma}
  \label{lemma:hom:est_sinc_discr}
  For $\Nquad \in \N$ and $\kquad \sim \Nquad^{-1/2}$, the following estimate holds for $0\leq \varepsilon < \min\big(\frac{\beta}{2},\frac{1}{4}\big)$:
  \begin{align*}
    \norm{\uu(t) - \uq(t)}_{\widetilde{H}^{\beta}(\Omega)}
    &\lesssim  t^{-\gamma/2} e^{-\omega \Nquad^{1/2} } \norm{u_0}_{H^{2\varepsilon}(\Omega)}
  \end{align*}
  for some constant $\omega>0$, depending on $\varepsilon$, $\gamma$, $\beta$.

\end{lemma}
\begin{proof}
  Using the error function
  \begin{align*}
    \mathcal{E}(\lambda,t):=\int_{-\infty}^{\infty}{g_{\lambda}(y,t) \, dy} - \kquad  \sum_{n=-\infty}^{\infty}{g_{\lambda}(n \,\kquad,t)}
  \end{align*}
  and the spectral decomposition $u_0=\sum_{j=0}^{\infty}{u_{0,j} \varphi_j}$, we can write the quadrature error as
  \begin{align*}
    \uu(t) - \uq(t) &=\sum_{j=0}^{\infty}{ \Big(\mathcal{E}(\lambda_j,t) +
                      \kquad \sum_{\abs{n} \geq \Nquad +1}{g_{\lambda_j}(n \, \kquad,t)} \Big) \,u_{0,j}\, \varphi_j}.
  \end{align*}
  For the error in the $\widetilde{H}^{\beta}(\Omega)$-norm, this means:
  \begin{align*}
    \norm{\uu(t) - \uq(t)}_{\widetilde{H}^{\beta}(\Omega)}^2
    &\lesssim \sum_{j=0}^{\infty}{  \lambda_j^\beta\abs{\mathcal{E}(\lambda_j,t)}^2 \abs{u_{0,j}}^2 +
      \Big( \kquad \sum_{\abs{n} \geq \Nquad +1}{ \lambda_j^{\beta/2} \abs{g_{\lambda_j}(n\,\kquad,t)}}\Big)^2 \abs{u_{0,j}}^2}.    
  \end{align*}
  The terms can be estimated by Proposition~\ref{prop:sinc_quad} and Lemma~\ref{lemma:hom:g_lambda_estimates}.
  \begin{align*}
    \norm{\uu(t) - \uq(t)}_{\widetilde{H}^{\beta}(\Omega)}^2
    &\lesssim
      t^{-\gamma} e^{-4 \pi d/k} \norm{u_0}_{\widetilde{H}^{2\varepsilon}(\Omega)}^2 +
      t^{-\gamma} \norm{u_0}_{\widetilde{H}^{2\varepsilon}(\Omega)}^2  \Big(\kquad \!\! \sum_{\abs{n} \geq \Nquad +1}{ e^{-\varepsilon n \kquad} \Big)^2} \\
    & \lesssim t^{-\gamma} \norm{u_0}_{\widetilde{H}^{2\varepsilon}(\Omega)}^2  \Big( e^{-4 \pi d/k} +  e^{- 2\varepsilon \Nquad \kquad} \Big).
  \end{align*}
  Picking $\kquad \sim \Nquad^{-1/2}$ and using Proposition~\ref{prop:no_bc_for_small_sobolev}  gives the stated result.
\end{proof}

\section{The inhomogeneous problem}
\label{sect:inhomogeneous_problem}
In this section, we focus on the inhomogeneous problem with homogeneous initial condition, i.e., \eqref{eq:model_problem} with $u_0=0$ and
general $f$.
The representation formula in this case reads
\begin{align}
  \uu_i(t)&=\int_{0}^{t}{\tau^{\gamma-1} e_{\gamma,\gamma}(-\tau^\gamma \LL^\beta) f(t-\tau) \,d\tau},
\end{align}
or, using the Riesz-Dunford calculus,
\begin{align}
  \uui(t)&=\frac{1}{2\pi \ii}\int_{0}^{t}{\tau^{\gamma-1} \left(\int_{\CC}{e_{\gamma,\gamma}(-\tau^\gamma z^\beta) R(z) f(t-\tau) \,dz} \right)\,d\tau}.
\end{align}
Our approximation scheme will be based on a sinc-type quadrature for the contour integral and an $hp$-type quadrature for the
convolution in time.

\subsection{$hp$-quadrature}
\label{sect:hp_quadrature}
In this section, we briefly summarize the theory and notation for using $hp$-methods to approximate integrals.
Given an interval $I=(a,b)$ and function $g: I \to \C$, we are interested in approximating
$$
\int_{I}{g(\tau) \,d\tau},
$$
where $g$ may have an algebraic singularity at the left endpoint $\tau=a$.

For a given degree $p \in \N_0$, we denote the Gauss quadrature points and weights on $(-1,1)$ by
$(x^p_j, w^p_j) \in (-1,1) \times \R_+$. See \cite[Sect.~{2.7}]{davis_rabinowitz} for details.
For $I=(-1,1)$, the  Gauss-quadrature approximation is then given by
\begin{align*}
  \QQ^p_{I} g:= \sum_{j=0}^{p}{w^p_j g(x^p_j)}.
\end{align*}
For general $I=(a,b)$, the approximation is obtained by an affine change of variables.

In order to get a method that adequately handles
a singularity at the left endpoint, we consider a mesh on the interval $(0,1)$ with a mesh grading factor $\sigma \in (0,1)$
and parameter $L \in \N$, $L\leq p$ given by
$$
K_0:=(0, \sigma^L),\, K_1:=(\sigma^{L}, \sigma^{L-1}),\, \dots,\, K_{L}:=(\sigma,1).
$$
On each one of these elements, we apply a Gauss quadrature, reducing the order as we approach the singularity, i.e.
\begin{align*}
  \int_{I}{g(\tau)\,d\tau}
  &\approx \QQ^{hp}_{I}
    g:= \sum_{\ell=0}^{L}{\QQ^{p-L+\ell}_{K_\ell} g}.
\end{align*}
For general intervals $(a,b)$ we again apply an affine change of variables.

The main result on such composite Gauss-quadrature is the following proposition, versions of which are well-known,  
see for example~\cite{schwab_quad,Schwab92,CvPS11}.
\begin{proposition}
  \label{prop:hp_quad_estimates}
  Fix $T>0$. Assume that $g: I:=(0,T) \to \C$ can be  extended holomorphically to a function $g: \mathcal{O} \to \C$ such that
  \begin{enumerate}[(i)]
  \item $\mathcal{O}$ contains a positive sector $\{ z \in \C: z\neq 0, \abs{\operatorname{Arg}{z}} < \delta \}$ for some constant $\delta > 0$; 
  \item
    there exist constants $\alpha \in [0,1)$ and $C_g>0$ such that
    \begin{align}
    \abs{g(z)} \leq C_g \abs{z}^{-\alpha}  \qquad \forall z \in \mathcal{O}.
    \label{eq:hp_quad_assumptions:bound}
    \end{align}
  \end{enumerate}

  Fix $\sigma > 0$ and consider the composite Gauss quadrature rule with $L$ layers of refinement with degree $p=L$.
  Then there exist positive constants $C_q$, $\omega$, and $C_{\text{stab}}$ such that
  \begin{align}
    \label{eq:hp_quad_est}
    \abs{ \int_{0}^{T}{g(\tau)\,d\tau} - \QQ_I^{hp} g}
    &\leq C_q T^{1-\alpha} e^{- \omega p} \qquad \text{and} \qquad
    \abs{\QQ_I^{hp} g} \leq C_{\text{stab}} \,\,T^{1-\alpha}. 
  \end{align}
\end{proposition}
\begin{proof}
  To prove the estimate~\eqref{eq:hp_quad_est}, we assume $T=1$. The more general result follows by an affine transformation.
  We distinguish two cases. On the element $K_0$, we use the bound on $g$ and the fact that $K_0$ is small to get
  \begin{align*}
    \abs{\int_{0}^{\sigma^L}{g(\tau)\,d\tau} - \QQ_{K_0}^{0} g}
    &\lesssim C_g\int_{0}^{\sigma^L}{\tau^{-\alpha} \,d\tau} + \sigma^L |g(\sigma^L/2)|
      \lesssim C_g \sigma^{(1-\alpha)L}.      
  \end{align*}
  For any other element $K_{L-\ell}=(\sigma^{\ell+1},\sigma^\ell)$, we map to the reference interval $(-1,1)$:
  \begin{align}
    \abs{\int_{\sigma^{\ell+1}}^{\sigma^\ell}{g(\tau)\,d\tau } - \QQ^{p-\ell}_{K_{\ell}} g}
    &\sim \sigma^{\ell} \abs{\int_{-1}^{1}{\widetilde{g}(\widetilde{\tau})\, d\widetilde{\tau}} - \QQ^{p-\ell}_{(-1,1)}\widetilde{g}},
      \label{eq:hp_quad_proof1}
  \end{align}
  where $\widetilde{g}:=g\big(\frac{\sigma^\ell}{2}{(1-\sigma)\tau + (1+\sigma)}\big)$ is the pull-back of $g$ to $(-1,1)$.
  By geometric considerations, $\widetilde{g}$ is analytic on an ellipse $\mathcal{E}_{\rho}$ with semiaxis sum $\rho>1$, where $\rho$ depends on
  $\sigma$ and the opening angle of the sector. We apply standard results on Gauss-quadrature (see~\cite[Eqn.~(4.6.1.11)]{davis_rabinowitz})
  to estimate:
  \begin{align*}
   \Big|{\int_{-1}^{1}{\widetilde{g}(\widetilde{\tau})\, d\widetilde{\tau}} - \QQ^{p-\ell}_{(-1,1)}\widetilde{g}}\Big|
    &\lesssim \rho^{-2(p-\ell)} \sup_{z\in \mathcal{E}_{\rho}}{\abs{\widetilde{g}(z)}}
      \lesssim C_{g}\rho^{-2(p-\ell)} \sigma^{-\alpha \ell},
  \end{align*}
  where we used Assumption~(\ref{eq:hp_quad_assumptions:bound}) to bound $\widetilde{g}$.
  Going back to~\eqref{eq:hp_quad_proof1} we conclude
  \begin{align*}
    \Big|{\int_{\sigma^{\ell+1}}^{\sigma^\ell}{g(\tau)\,d\tau} - \QQ^{hp} g}\Big|&\lesssim C_g \max(\sigma^{1-\alpha},\rho^{-2})^{p}.
  \end{align*}
  For general intervals $(0,T)$ we note that  $C_g$ behaves like $C_g \sim T^{-\alpha}$ under affine transformations.

  We now show the stability estimate.
  On each subinterval $K_{\ell}$, for $\ell>0$ we have
  \begin{align*}
    \sigma t < T\sigma^{L-\ell+1} \leq t  \leq T\sigma^{L-\ell}.
  \end{align*}
   Since all the  weights of Gauss-methods are positive and the quadrature integrates constants exactly, we can calculate  
  \begin{align*}
    \abs{\QQ^{hp}_{I} g}
    &\lesssim \QQ^{hp} \big(\tau^{-\alpha}\big)
    \lesssim
      \Big(\frac{T \sigma^{L}}{2}\Big)^{-\alpha}+
      \sum_{\ell=0}^{L-1}{\QQ^{p-\ell} (T\sigma^{\ell+1})^{-\alpha}}
      = \Big(\frac{T \sigma^{L}}{2}\big)^{-\alpha}+
      \sum_{\ell=0}^{L-1}{\int_{T \sigma^{\ell+1}}^{T\sigma^\ell}}{ \big(T\sigma^{\ell+1}\big)^{-\alpha} \, d\tau} \\
    &\lesssim \Big(\frac{T \sigma^{L}}{2}\big)^{-\alpha}+
      \sum_{\ell=0}^{L-1}{\int_{T \sigma^{\ell+1}}^{T\sigma^\ell}}{ (\sigma \tau)^{-\alpha} \, d\tau}
      \lesssim \big(\sigma^{-\alpha L} + \sigma^{-\alpha} \big)T^{1-\alpha}. \qedhere
  \end{align*}
    
\end{proof}

\subsection{Analysis of the inhomogeneous problem}
We are now in the position to define and analyze the discretization scheme for the inhomogeneous problem.
The discretization is done on multiple levels. We define the operator-valued functions
\begin{align*}
  W(\tau)&:=\tau^{\gamma-1} e_{\gamma,\gamma}(-\tau^\gamma\gamma \LL^{\beta})
           =\tau^{\gamma-1}\frac{1}{2\pi\ii}{\int_{\CC}{e_{\gamma,\gamma}(-\tau^\gamma z^\beta)\, R(z)\,dz}}, \\
  W^{q}(\tau)&:=\tau^{\gamma-1} \frac{\kquad}{2\pi \ii}\sum_{n=-\Nquad}^{\Nquad}
               {e_{\gamma,\gamma}(-\tau^\gamma z^\beta(n\,\kquad)) z^\prime(n\,\kquad) R(z(n\,\kquad))},\\
  W^{q,h}(\tau)&:=\tau^{\gamma-1} \frac{\kquad}{2\pi \ii}\sum_{n=-\Nquad}^{\Nquad}
               {e_{\gamma,\gamma}(-\tau^\gamma z^\beta(n\,\kquad)) z^\prime(n\,\kquad) R_h(z(n\,\kquad))}.
\end{align*}
We get the first layer of approximation by replacing the convolution in time with an $hp$-quadrature:
\begin{align*}
  \uiq(t):=\QQ^{hp}_{(0,t)} \left[W(\cdot) f(t-\cdot)\right].
\end{align*}
Instead of relying on the exact evaluation of $\LL^{\beta}$ we then switch to the sinc-quadrature approximation of the
Riesz-Dunford integral by defining 
$$
\uiqq(t):=\QQ^{hp}_{(0,t)}  \left[W^q(\cdot) f(t-\cdot)\right].
$$
Finally, we do  the discretization in space to obtain the fully discrete function
$$
\uiqqh(t):=\QQ^{hp}_{(0,t)}  \left[W^{q,h}(\cdot) f(t-\cdot)\right].
$$

We now step by step estimate the discretization errors, starting with the one in space.
\begin{lemma}
  \label{lemma:inh:est_space_discr}
  Assume that $f: (0,T) \to L^2(\Omega)$ is uniformly analytic (cf. Definition~\ref{def:uniformly_analytic}), and 
  $\HHx$ resolves the scales down to $\varepsilon = b^{-1/2} e^{- \Nquad \kquad/2 }$. Then
  \begin{align*}
    \norm{\uiqq(t)-\uiqqh(t)}_{\widetilde{H}^{\beta}(\Omega)}
    &\lesssim
      t^{\gamma/2} \kquad\, \Nquad\, e^{-\omega \Ndof^\mu}.
  \end{align*}
\end{lemma}
\begin{proof}
  The proof is very similar to Lemma~\ref{lemma:hom:est_space_discr}, we only have to account for
  the extra quadrature step.
  Since the proof only relied on the analyticity of the right-hand side $f$ and the estimate~\eqref{eq:mittag_leffler_est},
  we can analogously estimate for all $\tau \in (0,t)$ 
  \begin{align*}
    \norm{W^{q}(\tau)f(t-\tau) - W^{q,h}(\tau)f(t-\tau)}_{\widetilde{H}^{\beta}(\Omega)}
    &\lesssim
      \tau^{\gamma/2-1} \Nquad \kquad e^{-\omega \Ndof^\mu}.
  \end{align*}
  Using the stability estimate for $hp$-quadrature in~\eqref{eq:hp_quad_est}, the stated result follows.  
\end{proof}

Next, we consider the error due to the sinc-quadrature: 
\begin{lemma}
  \label{lemma:inh:est_sinc_discr}
  For $0<\varepsilon<\min\big(\frac{\beta}{2},\frac{1}{4}\big)$, $\Nquad \in \N$ and $\kquad \sim \Nquad^{-1/2}$ we can estimate
  \begin{align}
    \label{eq:inh_est_space_discr}
    \norm{\uiq(t)-\uiqq(t)}_{\widetilde{H}^{\beta}(\Omega)}
    &\lesssim
      t^{\gamma/2}  e^{-\omega \Nquad^{1/2}} \sup_{\tau \in (0,t)}{\norm{f(\tau)}_{H^{2\varepsilon}(\Omega)}}.
  \end{align}

\end{lemma}
\begin{proof}
  We again focus on pointwise estimates of $W-W^q$, this time proceeding analogously to Lemma~\ref{lemma:hom:est_sinc_discr}.
  We get for $\tau \in (0,t)$:
  \begin{align*}
      \norm{W(\tau)f(t-\tau) - W^{q}f(t-\tau)}_{\widetilde{H}^{\beta}(\Omega)}
    &\lesssim  \tau^{\gamma/2-1} e^{-\omega \Nquad^{1/2} } \norm{f(t-\tau)}_{H^{2\varepsilon}(\Omega)}.
  \end{align*}
  Using the stability of the $hp$-quadrature then gives~\eqref{eq:inh_est_space_discr}.
\end{proof}

The final step is analyzing the error due to $hp$-quadrature.
\begin{lemma}
  \label{lemma:inh:est_hp_quad}
  Assume that $f: [0,T] \to L^2(\Omega)$ is uniformly analytic.
  Assume we are using an $hp$-quadrature with $\Nhpquad$ refinement layers and
  maximum order $p=\Nhpquad$. 
  Then we can estimate
  \begin{align*}
    \norm{\uui(t)-\uiq(t)}_{\widetilde{H}^{\beta}(\Omega)}
  &\lesssim t^{\gamma/2} e^{-\omega \Nhpquad}.
  \end{align*}
  The constant depends on $f$ and the mesh grading factor $\sigma$ used for the $hp$-quadrature.
\end{lemma}
\begin{proof}
  Using the notation
    $W_{\lambda}(\tau):=\tau^{\gamma-1}e_{\gamma,\gamma}(-\tau^\gamma \lambda^\beta)$,
  we write via the  spectral decomposition
  \begin{align*}
    \uui(t)&=\sum_{j=0}^{\infty}{\int_{0}^{t}{ W_{\lambda_j}(\tau)\, \big(f(t-\tau),\varphi_j\big)_{L^2(\Omega)} \,d\tau}\, \varphi_j}, \\
    \uiq(t)&=\sum_{j=0}^{\infty}{\QQ^{hp}_{(0,t)}
           \Big( W_{\lambda_j}(\cdot)\, \big(f(t-\cdot),\varphi_j\big)_{L^2(\Omega)} \Big) \,\varphi_j}.
  \end{align*}
  By Definition~\ref{def:uniformly_analytic}~(\ref{it:uniformly_holomorphic_iii}),
    the integrands are analytic on  a set $\mathcal{O}$ containing a
    positive sector.
  Using Proposition~\ref{prop:hp_quad_estimates}, it is sufficient to show for all $\tau \in \mathcal{O}$ and $\lambda>0$ that
  $
  \abs{\lambda^{\beta/2} W_{\lambda}(\tau) } \leq \abs{\tau}^{\gamma/2-1}
  $
  with a constant independent of $\lambda$.
  We note that $-\tau^{\gamma} \lambda^\beta$ is always in the sector required for~\eqref{eq:mittag_leffler_est}.
  This can be seen since $-\tau^\gamma$ is always in the left-half plane with argument between $-\gamma \pi/4$ and $\gamma \pi/4$.
  Multiplying with $\lambda^\beta>0$ does not change the argument of the complex number.

  Therefore, we can estimate using~\eqref{eq:mittag_leffler_est}:
  \begin{align*}
    \abs{\lambda^{\beta/2} W_{\lambda}(\tau) }
    &\lesssim \abs{\tau}^{\gamma-1}  \frac{\lambda^{\beta/2}}{1+\abs{\tau}^\gamma \lambda^\beta}
      \lesssim \abs{\tau}^{\gamma-1}  \left(\frac{1}{1+\abs{\tau}^\gamma \lambda^\beta}\right)^{1/2}
      \left(\frac{\lambda^\beta}{1+\abs{\tau}^\gamma \lambda^\beta}\right)^{1/2} 
    \lesssim    \abs{\tau}^{\gamma/2-1}.
  \end{align*}
  Applying Proposition~\ref{prop:hp_quad_estimates} then concludes the proof.
\end{proof}

\section{The general result}
Combining the approximations for the  inhomogeneous and homogeneous scheme, we define the fully discrete approximation
as
$$
\uufd(t):=\uqh(t) + \uiqqh(t).
$$
The corresponding error analysis is then a simple combination of the individual error estimates.
Since we will use it later on, we also include a stability result with respect to the initial data. 

\begin{theorem}
  \label{thm:the_big_theorem}
  Assume that
  $u_0$ is analytic on a neighborhood of $\overline{\Omega}$ and assume that $f$ is uniformly analytic (cf. Definition~\ref{def:uniformly_analytic}).  
  Use $\Nquad \in \N$ quadrature points for the sinc-quadrature with $\kquad := \kappa \,\Nquad^{-1/2}$,
  $\Nhpquad\sim \sqrt{\Nquad}$ layers for the hp-quadrature. Assume that $\HHx$ resolves the scales down to
  $b^{-1/2} \, e^{-(\kappa/2) \sqrt{\Nquad}}$.
  Then, the following estimate holds:
  \begin{align*}
    \norm{\uu(t) - \uufd(t)}_{\widetilde{H}^{\beta}(\Omega)}
    &\lesssim \big(t^{-\gamma/2}+ t^{\gamma/2}\big)\,\big(e^{-\omega \sqrt{\Nquad}} + \sqrt{\Nquad} \,e^{-\omega \Ndof^\mu}\big).
  \end{align*}
  In addition, the following stability estimate holds for any $\varepsilon \in (0,1/2)$:
  \begin{align}
      \label{eq:fd_stabiltiy}
      \|\uufd(t)\|_{\widetilde{H}^{\beta}(\Omega)}
    &\lesssim C(\varepsilon)\,\min\big(\sqrt{\Nquad},t^{-\varepsilon \gamma}\big)
      \Big(t^{-\gamma/2}\,\|u_0\|_{L^2(\Omega)}
      +t^{\gamma/2}\max_{0\leq \tau \leq t}{\|f(\tau)\|_{L^2(\Omega)}}\Big).
    \end{align}
\end{theorem}
\begin{proof}
  For the convergence result, we just collect the different convergence results of Sections~\ref{sect:homogeneous_problem} and~\ref{sect:inhomogeneous_problem}.

  To prove the stability result, we start with only the homogeneous contribution.
  By an analogous computation to~\eqref{eq:bounding_the_integrand} we get from the definition of $\uqh$ and the estimate on  the resolvent~\eqref{eq:resolvent_estimates_l2}
  for $\varepsilon \in (0,1/2)$:
  \begin{align*}
    \norm{\uqh(t)}_{\widetilde{H}^{\beta}(\Omega)}
    &\lesssim k \sum_{n=-\Nquad}^{\Nquad}{ t^{-\gamma(1/2+\varepsilon)} \abs{z(y_n)}^{-\beta(1/2+\varepsilon)} \abs{z'(y_n)}\norm{R_h(y_n) u_{0}}_{\widetilde{H}^{\beta}(\Omega)} } \\
    &\lesssim  t^{-\gamma(1/2+\varepsilon)} \|u_{0}\|_{L^2(\Omega)}.
  \end{align*}
  From which the stated estimate follows readily via~\eqref{eq:zprime_vs_z}.
  If $\varepsilon=0$, the same computation can be made, but we end up with an additional factor $\Nquad k$
  compensating for the summation.
  The inhomogeneous contribution follows along the same lines, but using the stability
  of the $hp$-Quadrature~\eqref{eq:hp_quad_est}. 
\end{proof}

\subsection{Space-time and time robust estimates}
\label{sect:spacetime_and_time_robust_estimates}
Up to now, we have looked at the error in the $\widetilde{H}^{\theta}$-norm pointwise in time.
While such an approach is natural, the resulting estimates
deteriorate for small times $t$ close to zero like $\bigO(t^{-\gamma/2})$.
If the initial condition does not satisfy the boundary conditions, we cannot hope to derive $t$-robust estimates in the energy norm $\widetilde{H}^{\beta}$. In this section, we derive a different estimate which does not suffer from this deterioration.

For the following results to hold, we have to make an additional assumption on $\HHx$, a version of which is also already
present in~\cite{hp_for_heat}:
\begin{assumption}
  \label{ass:approx_of_unitial_condtiion}
  There exists a fixed neighborhood $\widetilde{\Omega}$ of $\overline{\Omega}$
    and constants $\theta$, $\omega$, $\mu>0$
  such that for all $u_0$ that are analytic on  $\widetilde{\Omega}$, 
  there exists a function $u_{h,0} \in \HHx$ and constants $C_{\text{stab}}$, $C_{\text{approx}}>0$ (depending on $u_0$) such that
  \begin{align*}
    \norm{u_{h,0}}_{\mathbb{V}_{h}^\theta} \leq C_{\text{stab}}
    \qquad \text{ and } \qquad \norm{u_{0} - u_{h,0}}_{L^2(\Omega)} \leq C_{\text{approx}} e^{- \omega \mathcal{N}_{\Omega}^\mu },
  \end{align*}
  where $\Ndof:=\operatorname{dim}\big(\HHx\big)$ and
  $
  \HHx^\theta:=\big[\big(\HHx, L^2(\Omega)\big), \big(\HHx, H^1(\Omega)\big) \big]_{\theta,2}.
  $
\end{assumption}

We start by refining the estimates on the resolvent operator
  from Proposition~\eqref{prop:resolvent_estimates_l2}, establishing that if we insert more regularity than $L^2$ in the
argument, we get improved damping properties.
\begin{lemma}
  \label{lemma:resolvent_estimates}
  Let $z \in \sector$.  Assume $u_0 \in \widetilde{H}^{\theta}(\Omega)$
  and $u_{h,0} \in \HHx^\theta$ for some $\theta \in [0,1]$. Then the following estimates hold for $\alpha \in [\theta,1]$:
  \begin{align}
    \label{eq:resolvent_estimates}
    \norm{R(z) u_0}_{\widetilde{H}^{\alpha}(\Omega)}
    &\lesssim \abs{z}^{-1+\frac{\alpha-\theta}{2}} \norm{u_0}_{\widetilde{H}^{\theta}(\Omega)} \quad  \text{and}\quad
    \norm{R_h(z) u_{h,0}}_{\widetilde{H}^{\alpha}(\Omega)}
    \lesssim \abs{z}^{-1+\frac{\alpha-\theta}{2}} \norm{u_{h,0}}_{\HHx^\theta}.
  \end{align}
  The implied constant depends on $\delta$, $A$, $c$, and the constants in the
  definition of $\sector$.
\end{lemma}
\begin{proof}
  We only consider the discrete case. The continuous one follows analogously, replacing $\HHx$ with $H^1_0(\Omega)$.
  We consider the function $w:=[R_h(z) - z^{-1}] u_{h,0}$. By elementary computations
    we get from~\eqref{eq:def_disc_resolvent} that
    $w$ solves for all $v \in \HHx$:
    \begin{align*}
      \big((z-c) w,v_h\big)_{L^2(\Omega)} - \big(A\nabla w,\nabla v_h\big)_{L^2(\Omega)}
      &=z^{-1}\big(\left(c  \, u_{h,0},v_h\right)_{L^2(\Omega)} + \left(A\nabla u_{h,0},\nabla v_h\right)_{L^2(\Omega)} \big) .
    \end{align*}
    Thus, $w$ solves the same variational problem as $R_h(z) u_{h,0}$ with modified right-hand side.
    If $u_{h,0} \in \HHx$, then $w \in \HHx$ is a valid test function, and we get
    fixing $\beta \in \C$ as in Proposition~\ref{prop:resolvent_estimates_l2} that
    \begin{multline*}
      \Re(\beta)\|A^{1/2}\nabla w\|^2_{L^2(\Omega)} - \Re(\beta z) \|w\|^2_{L^2(\Omega)} \\
    \lesssim
    |z|^{-1}\big(\|u_{h,0}\|_{L^2(\Omega)}\|w\|_{L^2(\Omega)} + \|\nabla u_{h,0}\|_{L^2(\Omega)}
    \|\nabla w\|_{L^2(\Omega)}\big).
  \end{multline*}
  From this and the already established bounds on $R_h(z) u_{h,0}$ from~\eqref{eq:resolvent_estimates_l2}, we readily get the estimates
  $$
  \|w\|_{H^1(\Omega)} \lesssim \abs{z}^{-1} \|u_{h,0}\|_{H^1(\Omega)}
  \quad \text{and}\quad
  \|w\|_{L^2(\Omega)} \lesssim \abs{z}^{-1} \|u_{h,0}\|_{L^2(\Omega)}.
  $$
  By interpolation, this gives
  \begin{align*}
    \|w\|_{H^{\alpha}(\Omega)} \lesssim \abs{z}^{-1} \|u_{h,0}\|_{\HHx^\alpha}.
  \end{align*}
  From the definition of $w$ and \eqref{eq:resolvent_estimates_l2}, we get the estimates:
  \begin{align*}
    \|R_h(z) u_{h,0} \|_{\widetilde{H}^{\alpha}(\Omega)} &\lesssim \abs{z}^{-1} \|u_{h,0}\|_{\HHx^\alpha} ,
                                         \quad \text{and}\quad
    \|R_h(z) u_{h,0}\|_{\widetilde{H}^{\alpha}(\Omega)} \lesssim \abs{z}^{-1+\alpha/2} \|u_{h,0}\|_{L^2(\Omega)}.
  \end{align*}
  By the reiteration theorem~\cite[Thm. 26.3]{tartar},
  we can write $\HHx^{\theta}=[L^2(\Omega),\HHx^{\alpha}]_{\frac{\theta}{\alpha}}.$
  Thus, we further interpolate the previous estimates to get for $\theta \in [0,\alpha]$:
 \begin{align*}
  \|R_h(z) u_{h,0} \|_{\widetilde{H}^{\alpha}(\Omega)} &\lesssim  
  \abs{z}^{-\frac{\theta}{\alpha} + (-1+\frac{\alpha}{2})(1-\frac{\theta}{\alpha})} \|u_{h,0}\|_{\HHx^\theta} ,
  =\abs{z}^{-1 + \frac{\alpha}{2}- \frac{\theta}{2}} \|u_{h,0}\|_{\HHx^\theta}.
  \qedhere
 \end{align*}

\end{proof}

We use this estimate to prove a time-robust estimate. We weaken the statement from an estimate which is pointwise in time to a space-time Sobolev norm.
Such norms were also considered in~\cite{hp_for_heat}.
\begin{theorem}
  \label{thm:spacetime_estimates}
  Let the Assumptions of Theorem~\ref{thm:the_big_theorem} hold. In addition, let
  Assumption~\ref{ass:approx_of_unitial_condtiion} hold, and use $u_{h,0}$ for the initial condition of the discrete method.

  Use $\Nquad \in \N$ quadrature points for the sinc-quadrature with $\kquad :=\kappa\, \Nquad^{-1/2}$,
  $\Nhpquad\sim \sqrt{\Nquad}$ layers for the hp-quadrature. Assume that $\HHx$ resolves the scales down to
  $b^{-1/2}\,e^{-(\kappa/2) \sqrt{\Nquad}}$.
  Then, the following estimate holds:
  \begin{align*}
    \int_{0}^{T}{t^{\gamma-1}
    \big\|{\uu(t) - \uufd(t)}\big\|_{\widetilde{H}^{\beta}(\Omega)}^2 \,dt}
    &\lesssim C(T)\big( e^{-\omega \sqrt{\Nquad}} + \sqrt{\Nquad}\,e^{-\omega \Ndof^\mu}\big).
  \end{align*}
  The constants depends on the end time $T$, the domain $\Omega$, the data $u_0$ and $f$, the constants from Assumption~\ref{ass:approx_of_HHx} and~\ref{ass:approx_of_unitial_condtiion},
  and on the details of the discretization, i.e., mesh grading, the  factor $\kappa$,
    and the ratio $\Nhpquad/\sqrt{\Nquad}$, but is independent of the accuracy parameters $\Nquad$, $\kquad$, $\Nhpquad$ or $\Ndof$.
\end{theorem}
\begin{proof}
  We only consider the homogeneous part, the other one is even simpler as the singular behavior
  for small times is not present. In this setting we have $\uufd=\uqh$.
  Let $t_0>0$ to be fixed later.
  In~\eqref{eq:fd_stabiltiy} we have seen that
  $\uufd$ depends continuously on the initial condition like
  $ t^{-\gamma/2} \sqrt{\Nquad} \norm{u_{h,0}}_{L^2(\Omega)}$.
  Replacing the discrete initial condition with $u_0$ allows us to apply Theorem~\ref{thm:the_big_theorem} and we get:
  \begin{align*}
    \int_{t_0}^{T}{t^{\gamma-1}
    \big\|{\uu(t) - \uufd(t)}\big\|_{\widetilde{H}^{\beta}(\Omega)}^2 \,dt}
    &\lesssim \log(t_0) \big(e^{-\omega \sqrt{\Nquad}} + \sqrt{\Nquad}\norm{u_0 - u_{h,0}}_{L^2(\Omega)}^2 + \sqrt{\Nquad}\,e^{-\omega \Ndof^\mu}\big).
  \end{align*}
  For $t<t_0$, it is easy to derive the stability estimates
  \begin{align*}
    \big\|{\uu(t)}\big\|_{\widetilde{H}^{\beta}(\Omega)}&\lesssim t^{-\gamma/2 +\varepsilon} \norm{u_0}_{H^{2\varepsilon}(\Omega)} \qquad \text{and}\qquad
    \big\|{\uqh(t)}\big\|_{\widetilde{H}^{\beta}(\Omega)}\lesssim t^{-\gamma/2 + \varepsilon}\norm{u_{0,h}}_{\HHx^{2\varepsilon}(\Omega)}
  \end{align*}
  from their representation formulas and the stability estimates on the resolvents in Lemma~\ref{lemma:resolvent_estimates}.

  Thus, the stated theorem follows if we pick $t_0 \sim e^{-(\omega/\varepsilon) \min(\sqrt{\Nquad},\, \Ndof^\mu)}$
  and absorb all polynomial terms into the exponential.  
\end{proof}
\begin{remark}
  In the case $\gamma=1$, the requirement that $u_{h,0}$ needs to be used for the discrete initial condition can be dropped. This 
  is because the space-time norm depends continuously on the $L^2$-norm of the initial condition.
\eremk
\end{remark}

\section{$hp$-FEM}
\label{sect:hp_fem}
In this section, we provide a construction for $\HHx$ that satisfies Assumption~\ref{ass:approx_of_HHx}. It is based on
an $hp$-finite element method on a suitably refined grid towards the boundary $\partial \Omega$. The construction is the same as in~\cite{hp_for_heat}
and has similarly also already appeared in~\cite{tensor_fem} in the context of stationary elliptic problems.
In \cite{tensor_fem_on_polygons}, the construction of~\cite{tensor_fem} is generalized to
polygonal domains.

For this, we need to make the following simplifying assumption throughout this section:

\begin{assumption}
  Let $\Omega \subset \R^d$ for $d=1,2$ have analytic boundary.
  Assume that $A$ and $c$ are  analytic on a neighborhood $\widetilde{\Omega} \supset \overline{\Omega}$.
\end{assumption}

Just as for the $hp$-quadrature in Section~\ref{sect:hp_quadrature}, the construction for $\HHx$ is based on a geometric grid that is 
refined towards the lower-dimensional manifolds where singularities are expected. In 1d, this means a geometric grid towards the end points,
for 2d we make the following definitions, following~\cite{MS98}, see also \cite{tensor_fem} and \cite[Def.~{2.4.4}]{melenk_book}.

We first introduce the (shape regular) reference mesh, used to resolve the geometry of $\Omega$.
\begin{definition}[reference mesh]
\label{def:reference-mesh}
  Let $\widehat{S}:=(0,1)^2$ be the reference square, and $\TT_{\Omega}:=\big\{K_i\big\}_{i=0}^{\abs{\TT_{\Omega}}}$  a mesh of curved quadrilaterals
  with bijective element maps $F_{K}: \overline{\widehat{S}} \to \overline{K}$ satisfying
  \begin{enumerate}[(M1)]
  \item The elements $K_i$ partition $\Omega$, i.e., $\bigcup_{K_i \in \TT} \overline{K_i} = \overline{\Omega}$;
  \item for $i\neq j$, $\overline{K}_i \cap \overline{K}_j$ is either empty, a vertex or an entire edge;
  \item the element maps $F_{K}:\widehat{S} \to K$ are analytic diffeomorphisms;
  \item the common edge of two neighboring elements $K_i$, $K_j$ has the same parametrization from both sides,
    i.e., if $\gamma_{ij}$ is the common edge with endpoints $P_1$, $P_2$, then for $P \in \gamma_{i,j}$ we have
    $$
    \operatorname{dist}(F_{K_i}^{-1}P,F_{K_i}^{-1} P_{\ell}))=\operatorname{dist}(F_{K_j}^{-1}P,F_{K_j}^{-1} P_{\ell}) \qquad \text{for } \ell=1,2.
    $$
  \end{enumerate}  
\end{definition}

In order to be able to resolve boundary layers on small scales, we now refine the reference mesh geometrically towards the boundary. This is captured
in the next definition.
\begin{definition}[anisotropic geometric mesh]
  \label{def:anisotropic_geometric_mesh}
  Let $\TT_{\Omega}$ be a reference mesh, and assume that $K_{i}$, $i=0,\dots, n < \abs{\TT_{\Omega}}$ are the elements at the boundary.
  Also assume that the left edge $e:=\{0\}\times (0,1)$ is mapped to $\partial \Omega$, i.e., $F_{K_i}(e) \subseteq \partial \Omega$
  and $F_{K_i}(\partial S \setminus \overline{e}) \cap \partial \Omega = \emptyset$ for  $i=0,\dots, n$.
  The remaining elements are taken to satisfy $\overline{K_i} \cap \partial \Omega = \emptyset$, $i=n+1,\dots, \abs{\TT_{\Omega}}$.

  For $L \in \N$ and a mesh grading factor $\sigma \in (0,1)$, we subdivide the reference square
  \begin{align*}
    \widehat{S}^0:=(0,\sigma^{L}) \times (0,1),\; \quad \widehat{S}^{\ell}:=(\sigma^{\ell},\sigma^{\ell-1})\times (0,1),\; \ell=1,\dots. L.
  \end{align*}

  The \emph{anisotropic geometric mesh} $\TT_{\Omega}^{L}$ is then given by the push-forwards of the refinements in the boundary region, plus the unrefined interior elements:
  $$
  \TT_{\Omega}^{L}:=\Big\{ F_{K_i}(\widehat{S}^{\ell}), \; \ell=0,\dots,L,\;\; i=0,\dots,n\big\} 
   \cup \bigcup_{i=n+1}^{\abs{\TT_{\Omega}}}{{\{} K_{i}{\}} }.
  $$  
\end{definition}

\begin{definition}
  In one dimension, for $\Omega=(-1,1)$, the reference mesh is given by the single element $\TT_{\Omega}:=\{(-1,1)\}$ and
  the anisotropic geometric mesh is given by the nodes
  \begin{align*}
    x_0&:=-1, \; x_i:=-1 + \sigma^{L - i +1}, \; i=1,\dots, L, \;\\
    x_i&:=1-\sigma^{i - L}, \; i=L+1,\dots, 2L, \;\quad  x_{2L+1}:=1.
  \end{align*}
  For general $\Omega=(a,b)$ it is given by an affine transformation of the mesh on $(-1,1)$.
\end{definition}

Using these meshes, we can now give an exemplary construction for $\HHx$, which satisfies Assumptions~\ref{ass:approx_of_HHx} and~\ref{ass:approx_of_unitial_condtiion}.
\begin{definition}[$\HHx$ via $hp$-FEM]
  Let $\TT_{\Omega}^L$ be an anisotropic geometric mesh refined towards $\partial \Omega$ and fix $p\in \N$.
  We write
    $
    \QQ^p:=\operatorname{span}_{0 \leq i_1, \,\dots,i_d\leq p }\big\{x_1^{i_1} \, \dots \,x_{d}^{i_d}\big\}
    $ for the space of tensor product polynomials and set
  \begin{align}
    \label{eq:def_hhx_for_fem}
    \HHx:=S^{p,1}_0(\TT^{L}_{\Omega}):=\Big\{ u \in H_0^{1}(\Omega): \; u\circ F_K \,\in\, \QQ^p \quad \forall K \in \TT_{\Omega}^{L} \big\}.
  \end{align}
\end{definition}

In~\cite{hp_for_heat}, it was shown that such spaces are able to reliably resolve small scales. We collect the result in the following proposition.
\begin{proposition}[{\cite[Thm.{3.33}]{hp_for_heat}}]
  \label{prop:hp_fem_resolves_scales}
  Let $\TT^{L}_{\Omega}$ be an anisotropic mesh on $\Omega$ that is geometrically refined towards $\partial \Omega$ with  grading factor $\sigma \in (0,1)$
  and $L$ layers. Use polynomial degree $p \sim L$.
  
  Then $\HHx$ defined in~\eqref{eq:def_hhx_for_fem} resolves the scales  down to $\sigma^{L}$, i.e.,
  there exist constants $C, \omega > 0$, such that
  for  $z \in \mathscr{S}$ with ${\abs{z}}^{-1/2}>\sigma^{L}$ 
  and  every $f$ that is analytic on a neighborhood $\widetilde{\Omega}$ of $\overline{\Omega}$, the solution $u_z$ to
  $(\LL - z)  u = z f$  can be approximated by $v_h \in \HHx$ satisfying
  \begin{align*}
    \abs{z}^{-1} \norm{\nabla u- \nabla v_h}_{L^2(\Omega)}^2 +\norm{u - v_h}^2_{L^2(\Omega)}
    &\leq C e^{- \omega' p} \leq C'  e^{-\omega \mathcal{N}_{\Omega}^{\frac{1}{d+1}}}.
  \end{align*}
  The constant $\omega$ depends only on $\sigma$, $\Omega$ and $\widetilde{\Omega}$. The constants $C$, $C'$ also depend
  on the constants of analyticity of $f$.
\end{proposition}

If we are interested in the estimate from Section~\ref{sect:spacetime_and_time_robust_estimates},
we also need that the $hp$-FEM space can approximate the initial condition in a way that is stable
in the discrete interpolation norm $\HHx^\theta$. Again following what was done in~\cite{hp_for_heat},
we split the construction into two steps, first computing the $H^1$-best approximation in a space without boundary  conditions (this can be done on the reference mesh)
and then performing a cutoff procedure to correct the boundary conditions on the anisotropic geometric mesh.
We first recall the properties of the cutoff operator.
\begin{proposition}[{\cite[Lem.~{3.35}]{hp_for_heat}}]
  \label{lemma:cutoff_operator}
  Let $\TT^L_{\Omega}$ denote an anisotropic geometric mesh with reference mesh $\TT_{\Omega}$.
  Given $\ell \in \N_0$, $\ell\leq L$, there exists a  linear operator $\mathscr{C}_{\ell}: \SS^{p,1}(\TT_{\Omega}) \to \SS^{p,1}_0(\TT^L_{\Omega})$ such that
  for $\theta \in [0,1/2)$
  \begin{align}
    \label{eq:stability_cutoff_operator}
    \norm{\mathscr{C}_{\ell} v}_{H^1(\Omega)} &\lesssim \sigma^{-\ell/2} \norm{v}_{L^2(\Omega)} + \norm{\nabla v}_{L^2(\Omega)} \quad\text{and}\quad   
                                                \norm{v - \mathscr{C}_{\ell} v}_{L^2(\Omega)} \lesssim \sigma^{{\theta\ell}} \norm{v}_{H^{\theta}(\Omega)}.
  \end{align}
\end{proposition}

\begin{lemma}
  \label{lemma:inital_condition_in_right_space}
  Let $u_0$ be analytic in a neighborhood $\widetilde{\Omega} \supset \overline{\Omega}$, and let $0\leq \theta <1/2$.  
  Then there exists a function $u_{h,0} \in \HHx$ :
  \begin{align}
    \label{eq:inital_condition_in_right_space}
    \norm{u_{h,0}}_{\mathbb{V}_{h}^\theta} \lesssim \norm{u_0}_{H^{1}(\Omega)} 
    \qquad \text{ and } \qquad \norm{u_{h,0} - u_0}_{L^2(\Omega)} \lesssim e^{-\omega' p }.
  \end{align}

  In other words, if the number of refinement layers $L \sim p$, then
  $\HHx$ satisfies Assumption~\ref{ass:approx_of_unitial_condtiion}   with $\mu:=1/(d+1)$.
\end{lemma}
\begin{proof}
  A similar result has appeared in~\cite[Lem.~{3.36}]{hp_for_heat}, although in a more complicated setting.
  Let $\Pi_{H^1}: H^1(\Omega) \to \SS^{p,1}(\TT_{\Omega})$ denote the best approximation operator
    in the $H^1(\Omega)$-norm  (note that we do not enforce boundary conditions).
    Set $\widehat{u}_0:=\Pi_{H^1} u_0$.
  We define $u_{h,0}:=\mathscr{C}_{L} \widehat{u}_0 \in \HHx$ and the piecewise constant function $v(t) \in \HHx$ as
  \begin{align*}
    v(t):=\begin{cases}
    \mathscr{C}_{L} \widehat{u}_0 & t \in (0, \sigma^L), \\
    \mathscr{C}_{\ell} \widehat{u}_0 & t \in (\sigma^{\ell},\sigma^{\ell-1}), \; \ell=1,\dots,L, \\
    0 & t>1. 
  \end{cases}
  \end{align*}
  For $\varepsilon >0 $ sufficiently small, we then estimate using  \eqref{eq:stability_cutoff_operator}:
  \begin{align*}
    \norm{u_{h,0}}^2_{\mathbb{V}_{h}^\theta}
    &\leq
      \int_{0}^{1}{t^{-2\theta} \big( \norm{\mathscr{C}_{L} \widehat{u}_{0} - v(t)}^2_{L^2(\Omega)} + t^2 \norm{v(t)}^2_{H^1(\Omega)}\big)\,\frac{dt}{t}}
      + \frac{1}{2\theta}\|\mathscr{C}_L \widehat{u}_0\|_{L^2(\Omega)}^2
    \\
    &\stackrel{{\mathmakebox[\widthof{=}]{\eqref{eq:stability_cutoff_operator}}}}{\lesssim}
      \Big(\!1+\!\sum_{\ell = 0}^{L-1}{\sigma^{(1-\varepsilon) \ell} \int_{\sigma^{\ell+1}}^{\sigma^{\ell}}{t^{-2\theta -1}dt}}\Big)  \norm{\widehat{u}_{0}}^2_{H^{1/2-\frac{\varepsilon}{2}}(\Omega)}
      + \Big(\sum_{\ell = 0}^{L-1}{\sigma^{-\ell} \int_{\sigma^{\ell+1}}^{\sigma^{\ell}}{t^{-2\theta + 1} dt}}\Big)  \norm{\widehat{u}_{0}}^2_{H^1(\Omega)}
\\
    &\lesssim \Big(\!1\!+\sum_{\ell = 0}^{L-1}{\sigma^{\ell - 2\theta \ell - \varepsilon \ell}} \Big) \norm{\widehat{u}_{0}}^2_{H^{1/2}(\Omega)} +
      \Big(\sum_{\ell = 0}^{L-1}{\sigma^{-\ell + 2 \ell - 2\theta \ell}}\Big)  \norm{\widehat{u}_{0}}^2_{H^1(\Omega)}
    \lesssim \norm{\widehat{u}_0}_{H^1(\Omega)}^2,
  \end{align*}
  where in the last step we used the fact that $2\theta < 1$ and a geometric series.

  The statement then follows from the fact that the space  $\SS^{p,1}(\TT_{\Omega})$ can approximate analytic functions exponentially fast (see for example~\cite[Prop.~{3.2.21}]{melenk_book}),
  and the stability of the $H^1$-best approximation operator $\Pi_{H^1}$.  
\end{proof}

\begin{remark}
  The use of the $H^1$-norm on the right hand side of~\eqref{eq:inital_condition_in_right_space} is mostly for convenience. Using more involved
  results about interpolation spaces of polynomials, we expect that this requirement can be lowered to the $H^{\beta+\varepsilon}(\Omega)$-norm.
\eremk
\end{remark}
Collecting the previous results, we arrive at the pointwise error estimate:
\begin{corollary}
  \label{hp_pointwise_estimates}
  Assume that $u_0$ is analytic on a neighborhood of $\overline{\Omega}$ and assume that $f$ is uniformly analytic.

  Use $\Nquad \in \N$ quadrature points for the sinc-quadrature with $\kquad \sim 1/\sqrt{\Nquad}$,
  $\Nhpquad\sim \sqrt{\Nquad}$ layers and $p = \Nhpquad$ for the $hp$-quadrature.
  Let $\TT_{\Omega}^L$ be an anisotropic geometric grid with grading factor $\sigma \in (0,1)$ such that
  $L \gtrsim \sqrt{\Nquad}$ and use the polynomial degree $p \sim L$
    for $\HHx$.

  Then, the following estimate holds for a constant $\omega  >0$:
  \begin{align*}
    \norm{\uu(t) - \uufd(t)}_{\widetilde{H}^{\beta}(\Omega)}
    &\lesssim \big(t^{-\gamma/2}+ t^{\gamma/2}\big)\,\big(e^{-\omega \sqrt{\Nquad}} + \,e^{-\omega \Ndof^{\frac{1}{d+1}}}\big), 
  \end{align*}  
\end{corollary}
\begin{proof}
  We apply Theorem~\ref{thm:the_big_theorem}. By Proposition~\ref{prop:hp_fem_resolves_scales},
  the space $\HHx$ satisfies the necessary Assumption~\ref{ass:approx_of_HHx}.
  The factor $\sqrt{\Nquad}$ from Theorem~\ref{thm:the_big_theorem} is absorbed into the exponential for convenience.
\end{proof}

The estimate for the space-time energy norm takes the following form:
\begin{corollary}
  Fix $T>0$, assume that
  $u_0$ is analytic on a neighborhood $\widetilde{\Omega}$ of $\overline{\Omega}$ and assume that $f$ is uniformly analytic.
  Use $\Nquad \in \N$ quadrature points for the sinc-quadrature with $\kquad \sim 1/\sqrt{\Nquad}$,
  $\Nhpquad\sim \sqrt{\Nquad}$ layers and degree $p\sim \Nhpquad$ for the hp-quadrature.
  Let $\TT_{\Omega}^L$ be an anisotropic geometric grid with grading factor $\sigma \in (0,1)$ such that
  $L \gtrsim \frac{\sqrt{\Nquad}}{\abs{\ln(\sigma)}}$,
   and use the polynomial degree $p \sim L$
    for $\HHx$.
  
  Let $u_{h,0}$ be given by $u_{h,0}:=\CC_{L} \Pi_{H^1} u_0$ where $\CC_{L}$ is the cutoff operator from Proposition~\ref{lemma:cutoff_operator} and
  $\Pi_{H^1}$ is the $H^1$-orthogonal projection onto the space $\SS^{p,1}(\TT_\Omega)$.
  
  Then, the following estimate holds for a constant $\omega >0$:
  \begin{align*}
    \int_{0}^{T}{t^{\gamma-1}
    \big\|{\uu(t) - \uufd(t)}\big\|_{\widetilde{H}^{\beta}(\Omega)}^2 \,dt}
    &\lesssim C(T) \big( e^{-\omega \sqrt{\Nquad}} + e^{-\omega \Ndof^{\frac{1}{d+1}}}\big).
  \end{align*}
  The constants depends on the end time $T$, the domain $\Omega$, $\widetilde{\Omega}$, $\mathcal{O}$, the data $u_0$ and $f$, the constants from Assumption~\ref{ass:approx_of_HHx} and~\ref{ass:approx_of_unitial_condtiion}
  and on the details of the discretization like mesh grading or the ratio $\kquad/\sqrt{\Nquad}$, but is independent of the accuracy parameters $\Nquad$, $\kquad$, $\Nhpquad$ or $\Ndof$.
\end{corollary}
\begin{proof}
  Follows from~\ref{thm:spacetime_estimates}. The necessary assumptions on $\HHx$ and $u_{h,0}$ are satisfied by Proposition~\ref{prop:hp_fem_resolves_scales} and Lemma~\ref{lemma:inital_condition_in_right_space}.
\end{proof}

\section{Numerical Results}
In order to confirm our theoretical findings, we implemented the method using the NGSolve software package~\cite{ngsolve,ngsolve2} for the finite element discretization in $\Omega$.
The geometric meshes in 2d were generated using the $hp$-refinement feature of Netgen, the integrated mesh generator of NGSolve. A sample mesh can be seen in Figure~\ref{fig:mesh}. We note that these meshes
include geometric refinements towards each of the corners of the domain.

\begin{figure}
  \center  
  \includeTikzOrEps{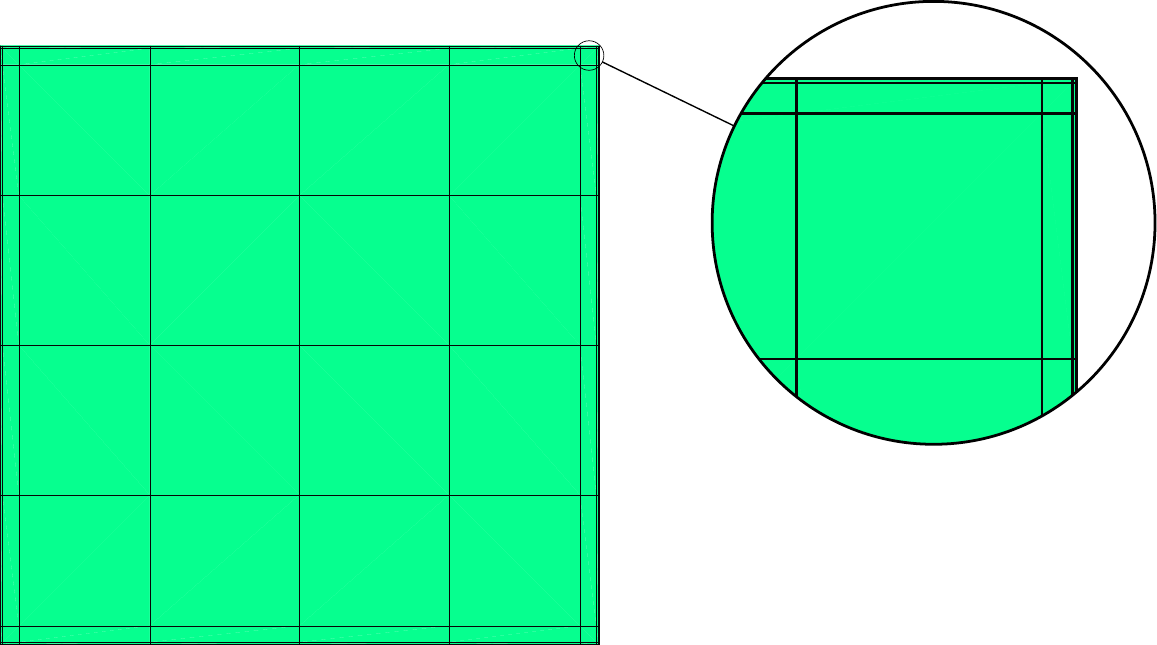}
  \caption{Example of a geometric mesh with 5 refinement layers.}
\label{fig:mesh}
\end{figure}

\begin{remark}
When implementing the presented method in practice, the dominant cost is to numerically solve the singularly perturbed problems $(z_j-\LL)^{-1}$.
Observing that these problems only depend on the quadrature point, and therefore appear several times throughout the algorithm, most notably inside the $hp$-quadrature for the time convolution,
it is therefore beneficial to reorder the operations in order to minimize the number of systems that need to be solved.
This leads to a method which only requires $\Nquad$ linear system solves, making it very efficient.
\eremk
\end{remark}

\begin{example}
  \label{example_smooth}
  As a first example, we consider the unit square $\Omega=(0,1)^2$ and fix the end time $t=1$.
  We set the parameters $\beta:=0.75$ and $\gamma:=0.6$.
  In order to validate our implementation, we choose the initial condition and right-hand side in a way that gives rise to a known exact solution.
  Namely, similarly to \cite{bonito_pasciak_parabolic}, we set $u(t,x,y):=e_{\gamma,1}(-t^{\gamma} (8\pi^2)^{\beta})\sin(2\pi x) \sin(2\pi y)+t^{3} \sin(\pi x) \sin(\pi y)$. This means the data are given by
  \begin{align*}
    u_0(x,y)&:=\sin(2\pi x) \sin(2\pi y), \quad \text{and}\quad
    f(t,x,y):=\Bigg(\frac{\Gamma(4)}{\Gamma(4-\gamma)} t^{3-\gamma} + t^{3}(2\pi^2)^{\beta} \Bigg) \sin(\pi x) \sin(\pi y).    
  \end{align*}  
  We look at the convergence of the $L^2(\Omega)$-error as we increase the polynomial degree $p$ used for our finite element discretization,
  the number of geometric refinements in the underlying grid as well as our
  quadrature parameters $\Nquad:=6\, p^2$, $\kquad:=\pi\sqrt{1/(5 \beta \Nquad)} \sim 1/p$
  and $\Nhpquad:=p$. The grading factor $\sigma:=0.125$ is used for both the geometric
  mesh on $\Omega$ and the $hp$-quadrature.

  Since we are working in $2d$ and we are dealing with a geometry with corners
    that need to be resolved,
    the number of degrees of freedom scale like $\Ndof\sim p^4$. We plot
  the error compared to $\Ndof^{1/4}$. As expected, we observe exponential convergence in this very simple case of a known smooth exact solution; see Figure~\ref{fig:smooth_exact_solution}.
\end{example}
\begin{figure}
  \begin{subfigure}{0.5\textwidth}
  \includeTikzOrEps{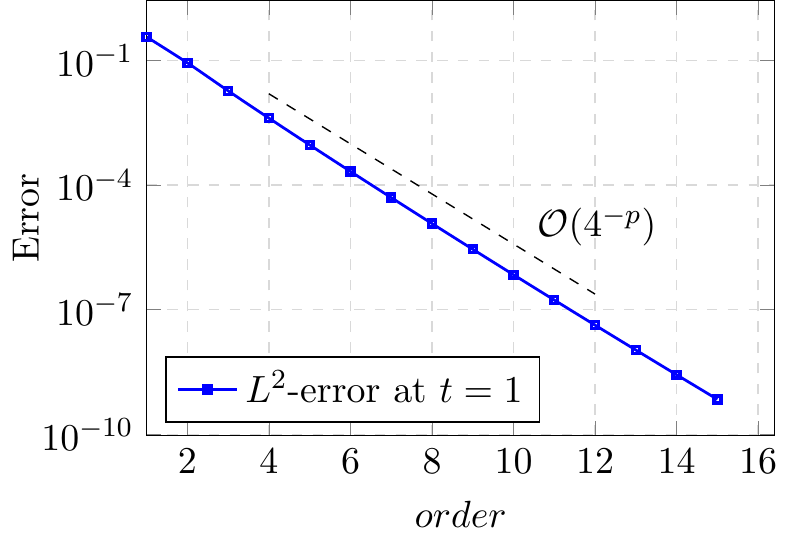}
  \caption{Compatible data; Example~\ref{example_smooth}}
  \label{fig:smooth_exact_solution}
\end{subfigure}
\begin{subfigure}{0.5\textwidth}
    \includeTikzOrEps{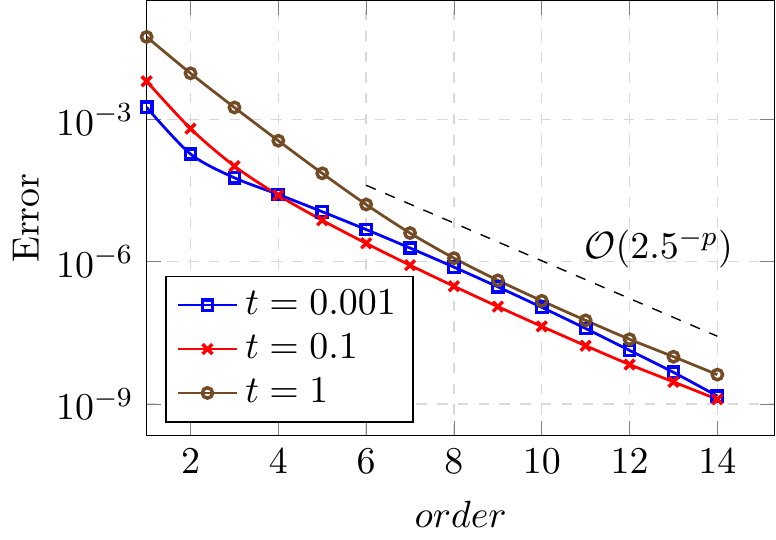}
  \caption{Incompatible data; Example~\ref{example_sing}}
  \label{fig:singular_exact_solution}
\end{subfigure}
\caption{Convergence at $t=1$ in the $L^2(\Omega)$-norm for compatible and incompatible data}
\end{figure}

\begin{remark}
  The choice of $\kquad$ is motivated by~\cite[Rem.~{4.1}]{blp17},
  where the optimal choice is given by
  $
  k=\sqrt{\frac{\pi H}{\beta N}},
    $
    where $H<\pi/4$ is determined by the region of analyticity in Lemma~\ref{lemma:hom:g_lambda_estimates}.
    In our case, we used $H:=\frac{\pi}{5}$.
\eremk
  \end{remark}

\begin{example}
  \label{example_sing}
  We continue with the unit square as the computational domain.
  But now, we take an initial condition and right-hand side that violates the boundary condition. This will lead to the formation of singularities.
  Namely, we fix $u_0:\equiv 1$ and $f(x,y,t):=\sin(t)$, and choose $\gamma=\sqrt{2}/2$ and $\beta:=\sqrt{3}/3$.
  We plot the convergence of the $L^2(\Omega)$-error as we increase the polynomial degree $p$ for different end times $t$. Again, we keep the other
  discretization parameters proportional, this time using
    $\Nquad:=6\, p^2$, $\kquad:=\pi\sqrt{1/(5 \beta \Nquad)} \sim 1/p$, $\Nhpquad:=p$, and
    again setting $\sigma:=0.125$ for all geometric grids.
  Since the exact solution is not known, we used the approximation on the finest grid as our reference solution.

  As expected, we again observe exponential convergence
  with respect to $\Ndof^{1/3}\sim p$, confirming that our numerical method can resolve all the appearing singularities.
\end{example}

\textbf{Acknowledgments:}
The authors gladly acknowledge financial support by the Austrian Science Fund (FWF)
through the projects SFB65~(A.R. and J.M.M) and P29197-N32 (A.R).

\bibliographystyle{alphaabbr}
\bibliography{literature}

\newcommand{\etalchar}[1]{$^{#1}$}
\begin{thebibliography}{BMN{\etalchar{+}}18}

\bibitem[AB17]{ab17}
G.~Acosta and J.~P. Borthagaray.
\newblock A fractional {L}aplace equation: regularity of solutions and finite
  element approximations.
\newblock {\em SIAM J. Numer. Anal.}, 55(2):472--495, 2017.

\bibitem[BBN{\etalchar{+}}18]{bbnos18}
A.~Bonito, J.~P. Borthagaray, R.~H. Nochetto, E.~Ot\'{a}rola, and A.~J.
  Salgado.
\newblock Numerical methods for fractional diffusion.
\newblock {\em Comput. Vis. Sci.}, 19(5-6):19--46, 2018.

\bibitem[BLP17a]{blp17}
A.~Bonito, W.~Lei, and J.~E. Pasciak.
\newblock The approximation of parabolic equations involving fractional powers
  of elliptic operators.
\newblock {\em J. Comput. Appl. Math.}, 315:32--48, 2017.

\bibitem[BLP17b]{bonito_pasciak_parabolic}
A.~Bonito, W.~Lei, and J.~E. Pasciak.
\newblock Numerical {A}pproximation of {S}pace-{T}ime {F}ractional {P}arabolic
  {E}quations.
\newblock {\em Comput. Methods Appl. Math.}, 17(4):679--705, 2017.

\bibitem[BMN{\etalchar{+}}18]{tensor_fem}
L.~Banjai, J.~M. Melenk, R.~H. Nochetto, E.~Ot{\'a}rola, A.~J. Salgado, and
  C.~Schwab.
\newblock Tensor fem for spectral fractional diffusion.
\newblock {\em Foundations of Computational Mathematics}, Oct 2018.

\bibitem[BMS20]{tensor_fem_on_polygons}
L.~Banjai, J.~M. Melenk, and C.~Schwab.
\newblock Exponential convergence of $hp$ fem for spectral fractional diffusion
  in polygons.
\newblock Preprint, Arxiv, https://arxiv.org/abs/2011.05701, 2020.

\bibitem[CvPS11]{CvPS11}
A.~Chernov, T.~von Petersdorff, and C.~Schwab.
\newblock Exponential convergence of {$hp$} quadrature for integral operators
  with {G}evrey kernels.
\newblock {\em ESAIM Math. Model. Numer. Anal.}, 45(3):387--422, 2011.

\bibitem[DR84]{davis_rabinowitz}
P.~J. Davis and P.~Rabinowitz.
\newblock {\em Methods of numerical integration}.
\newblock Computer Science and Applied Mathematics. Academic Press, Inc.,
  Orlando, FL, second edition, 1984.

\bibitem[DS21]{danczul2020reduced}
T.~Danczul and J.~Sch\"{o}berl.
\newblock A reduced basis method for fractional diffusion operators {II}.
\newblock {\em J. Numer. Math.}, 29(4):269--287, 2021.

\bibitem[DS22]{danczul2019reduced}
T.~Danczul and J.~Sch\"{o}berl.
\newblock A reduced basis method for fractional diffusion operators {I}.
\newblock {\em Numer. Math.}, 151(2):369--404, 2022.

\bibitem[FMMS21]{fmms21}
M.~Faustmann, C.~Marcati, J.~M. Melenk, and C.~Schwab.
\newblock {W}eighted analytic regularity for the integral fractional
  {L}aplacian in polygons.
\newblock Preprint, Arxiv, https://arxiv.org/abs/2112.08151, 2021.

\bibitem[GHK05]{GHK_05}
I.~P. Gavrilyuk, W.~Hackbusch, and B.~N. Khoromskij.
\newblock Hierarchical tensor-product approximation to the inverse and related
  operators for high-dimensional elliptic problems.
\newblock {\em Computing}, 74(2):131--157, 2005.

\bibitem[HHT08]{HHT_08}
N.~Hale, N.~J. Higham, and L.~N. Trefethen.
\newblock Computing {${\mathbf{A}}^\alpha,\ \log({\mathbf{A}})$}, and related
  matrix functions by contour integrals.
\newblock {\em SIAM J. Numer. Anal.}, 46(5):2505--2523, 2008.

\bibitem[HLM{\etalchar{+}}18]{HLMMV_18}
S.~Harizanov, R.~Lazarov, S.~Margenov, P.~Marinov, and Y.~Vutov.
\newblock Optimal solvers for linear systems with fractional powers of sparse
  {SPD} matrices.
\newblock {\em Numer. Linear Algebra Appl.}, 25(5):e2167, 24, 2018.

\bibitem[HLM{\etalchar{+}}20]{HLMMP_20}
S.~Harizanov, R.~Lazarov, S.~Margenov, P.~Marinov, and J.~Pasciak.
\newblock Analysis of numerical methods for spectral fractional elliptic
  equations based on the best uniform rational approximation.
\newblock {\em J. Comput. Phys.}, 408:109285, 21, 2020.

\bibitem[Hof20]{ClHofreit}
C.~Hofreither.
\newblock A unified view of some numerical methods for fractional diffusion.
\newblock {\em Comput. Math. Appl.}, 80(2):332--350, 2020.

\bibitem[KST06]{ksh06}
A.~A. Kilbas, H.~M. Srivastava, and J.~J. Trujillo.
\newblock {\em Theory and applications of fractional differential equations},
  volume 204 of {\em North-Holland Mathematics Studies}.
\newblock Elsevier Science B.V., Amsterdam, 2006.

\bibitem[LB92]{sinc_book}
J.~Lund and K.~L. Bowers.
\newblock {\em Sinc methods for quadrature and differential equations}.
\newblock Society for Industrial and Applied Mathematics (SIAM), Philadelphia,
  PA, 1992.

\bibitem[McL00]{McLean2000}
W.~McLean.
\newblock {\em Strongly elliptic systems and boundary integral equations}.
\newblock Cambridge University Press, Cambridge, 2000.

\bibitem[Mel02]{melenk_book}
J.~M. Melenk.
\newblock {\em {$hp$}-finite element methods for singular perturbations},
  volume 1796 of {\em Lecture Notes in Mathematics}.
\newblock Springer-Verlag, Berlin, 2002.

\bibitem[MPSV18]{mpsv18}
D.~Meidner, J.~Pfefferer, K.~Sch\"{u}rholz, and B.~Vexler.
\newblock {$hp$}-finite elements for fractional diffusion.
\newblock {\em SIAM J. Numer. Anal.}, 56(4):2345--2374, 2018.

\bibitem[MR21]{hp_for_heat}
J.~M. Melenk and A.~Rieder.
\newblock {$hp$}-{FEM} for the fractional heat equation.
\newblock {\em IMA J. Numer. Anal.}, 41(1):412--454, 2021.

\bibitem[MS98]{MS98}
J.~M. Melenk and C.~Schwab.
\newblock {$HP$} {FEM} for reaction-diffusion equations. {I}. {R}obust
  exponential convergence.
\newblock {\em SIAM J. Numer. Anal.}, 35(4):1520--1557, 1998.

\bibitem[NOS15]{nos15}
R.~H. Nochetto, E.~Ot\'{a}rola, and A.~J. Salgado.
\newblock A {PDE} approach to fractional diffusion in general domains: a priori
  error analysis.
\newblock {\em Found. Comput. Math.}, 15(3):733--791, 2015.

\bibitem[NOS16]{noes16}
R.~H. Nochetto, E.~Ot\'{a}rola, and A.~J. Salgado.
\newblock A {PDE} approach to space-time fractional parabolic problems.
\newblock {\em SIAM J. Numer. Anal.}, 54(2):848--873, 2016.

\bibitem[Rie20]{de_quad}
A.~Rieder.
\newblock Double exponential quadrature for fractional diffusion.
\newblock Preprint, Arxiv, https://arxiv.org/abs/2012.05588, 2020.

\bibitem[Sch92]{Schwab92}
C.~Schwab.
\newblock A note on variable knot, variable order composite quadrature for
  integrands with power singularities.
\newblock In {\em Numerical integration ({B}ergen, 1991)}, volume 357 of {\em
  NATO Adv. Sci. Inst. Ser. C: Math. Phys. Sci.}, pages 343--347. Kluwer Acad.
  Publ., Dordrecht, 1992.

\bibitem[Sch94]{schwab_quad}
C.~Schwab.
\newblock Variable order composite quadrature of singular and nearly singular
  integrals.
\newblock {\em Computing}, 53(2):173--194, 1994.

\bibitem[Sch14]{ngsolve}
J.~Sch{\"o}berl.
\newblock C++11 implementation of finite elements in ngsolve.
\newblock In {\em ASC Report 30/2014}. Institute for Analysis and Scientific
  Computing, Vienna University of Technology, 2014.

\bibitem[Sch22]{ngsolve2}
J.~Sch{\"o}berl.
\newblock Ngsolve.
\newblock Website, \url{ngsolve.org}, 2022.

\bibitem[SZB{\etalchar{+}}18]{collection_of_applications}
H.~Sun, Y.~Zhang, D.~Baleanu, W.~Chen, and Y.~Chen.
\newblock A new collection of real world applications of fractional calculus in
  science and engineering.
\newblock {\em Communications in Nonlinear Science and Numerical Simulation},
  64:213 -- 231, 2018.

\bibitem[Tar07]{tartar}
L.~Tartar.
\newblock {\em An introduction to {S}obolev spaces and interpolation spaces},
  volume~3 of {\em Lecture Notes of the Unione Matematica Italiana}.
\newblock Springer, Berlin; UMI, Bologna, 2007.

\bibitem[Tho06]{thomee_book}
V.~Thom\'ee.
\newblock {\em Galerkin finite element methods for parabolic problems},
  volume~25 of {\em Springer Series in Computational Mathematics}.
\newblock Springer-Verlag, Berlin, second edition, 2006.

\bibitem[Tri99]{Triebel1999}
H.~Triebel.
\newblock {\em Interpolation Theory - Function Spaces - Differential
  Operators}.
\newblock Wiley, 1999.

\bibitem[Tri06]{triebel3}
H.~Triebel.
\newblock {\em Theory of function spaces. {III}}, volume 100 of {\em Monographs
  in Mathematics}.
\newblock Birkh\"auser Verlag, Basel, 2006.

\end{thebibliography}

\end{document}